\documentclass[a4paper,11pt]{amsart}

\usepackage{amssymb}
\usepackage{mathrsfs}
\usepackage{url}

\newtheorem{theorem}{Theorem}[section]
\newtheorem{lemma}[theorem]{Lemma}
\newtheorem{corollary}[theorem]{Corollary}
\newtheorem{proposition}[theorem]{Proposition}
\newtheorem*{clemma}{Cassels lemma}

\theoremstyle{definition}
\newtheorem{remark}[theorem]{Remark}
\newtheorem{example}[theorem]{Example}

\numberwithin{equation}{section}

\makeatletter
\@namedef{subjclassname@2020}{%
  \textup{2020} Mathematics Subject Classification}
\makeatother

\DeclareMathOperator{\RE}{Re}

\DeclareMathOperator{\ord}{ord}
\DeclareMathOperator{\supp}{supp}
\DeclareMathOperator{\meas}{meas}

\begin{document}

\title[Hurwitz zeta-function with algebraic parameter]
{A random variable related to the Hurwitz zeta-function with algebraic parameter}

\author[M. Mine]{Masahiro Mine}
\address{Global Education Center\\ Waseda University\\
1-6-1 NishiWaseda, Shinjuku-ku, Tokyo 169-8050, Japan}
\email{m-mine@aoni.waseda.jp}

\date{}

\begin{abstract}
In this paper, we introduce a certain random variable closely related to the value-distribution of the Hurwitz zeta-function with algebraic parameter. 
We prove a version of the limit theorem, where the limit measure is presented by the law of this random variable. 
Then we apply it to show that any complex number can be approximated by values of the Hurwitz zeta-function for algebraic irrational parameters but with finite exceptions. 
\end{abstract}



\maketitle

\section{Introduction}\label{sec:1}
Let $s=\sigma+it$ be a complex variable. 
For a real number $\alpha$ satisfying $0<\alpha \leq1$, the Hurwitz zeta-function $\zeta(s,\alpha)$ is defined as 
\begin{gather}\label{eq:10060143}
\zeta(s,\alpha)
= \sum_{n=0}^{\infty} \frac{1}{(n+\alpha)^s}. 
\end{gather}
This series is convergent absolutely for $\sigma>1$. 
The Riemann zeta-function $\zeta(s)$ is a special case of the Hurwitz zeta-function. 
Indeed, we obviously have $\zeta(s,1)=\zeta(s)$. 
Some properties on $\zeta(s)$ are generalized to $\zeta(s,\alpha)$ for every $\alpha$. 
For example, $\zeta(s,\alpha)$ can be continued to a holomorphic function on $\mathbb{C}$ except only for a simple pole at $s=1$. 
On the other hand, Davenport--Heilbronn \cite{DavenportHeilbronn1936a} proved that $\zeta(s,\alpha)$ has infinitely many zeros for $\sigma>1$ if $\alpha\neq1/2$ or $1$ is rational or transcendental, while $\zeta(s)$ has no zeros for $\sigma>1$ due to the Euler product representation
\begin{gather}\label{eq:10032309}
\zeta(s)
= \prod_{p} \left(1-\frac{1}{p^s}\right)^{-1},  
\end{gather}
where $p$ runs through all prime numbers. 
The absence of the Euler product for the Hurwitz zeta-function causes significant differences in methods for investigating the value-distributions of $\zeta(s)$ and $\zeta(s,\alpha)$. 

Let $\mathbb{X}(p)$ be random variables indexed by prime numbers $p$, which are independent and uniformly distributed on the unit circle of the complex plane. 
According to \eqref{eq:10032309}, we define the random variable $\zeta(\sigma,\mathbb{X})$ as
\begin{gather*}
\zeta(\sigma,\mathbb{X})
= \prod_{p} \left(1-\frac{\mathbb{X}(p)}{p^\sigma}\right)^{-1}. 
\end{gather*}
This infinite product is convergent for $\sigma>1/2$ almost surely. 
Due to the unique factorization of the integers, the complex numbers $p^{-it}$ for $t \in \mathbb{R}$ behave as if they are independent random variables. 
Therefore we expect that the value-distribution of $\zeta(\sigma+it)$ in the $t$-aspect is approximated by using $\zeta(\sigma,\mathbb{X})$. 
In fact, we have the following limit theorem proved essentially by Jessen--Wintner \cite{JessenWintner1935}. 
Let $\sigma>1/2$ be a fixed real number. 
Then the probability measure
\begin{gather*}
P_{\sigma,T}(A)
= \frac{1}{T} \meas \left\{t \in [0,T] ~\middle|~ \zeta(\sigma+it) \in A \right\}, 
\quad
A \in \mathcal{B}(\mathbb{C})
\end{gather*}
converges weakly as $T \to\infty$ to the law of $\zeta(\sigma,\mathbb{X})$. 
Here, $\meas S$ stands for the one-dimensional Lebesgue measure of a set $S$, and $\mathcal{B}(\mathbb{C})$ is the algebra of Borel sets of $\mathbb{C}$. 
See Laurin\v{c}ikas \cite{Laurincikas1996a} for more information about the limit theorem for $\zeta(s)$. 

For the Hurwitz zeta-function, we consider another random variable to obtain the limit theorem. 
For $0<\alpha \leq1$, we define 
\begin{gather}\label{eq:10010000}
\zeta(\sigma,\mathbb{X}_\alpha)
= \sum_{n=0}^{\infty} \frac{\mathbb{X}_\alpha(n)}{(n+\alpha)^\sigma}, 
\end{gather}
where $\mathbb{X}_\alpha(n)$ are some random variables indexed by integers $n \geq0$. 
Our first purpose is to suitably choose $\mathbb{X}_\alpha(n)$ and to show that the probability measure
\begin{gather}\label{eq:10010305}
P_{\sigma,\alpha,T}(A)
= \frac{1}{T} \meas \left\{t \in [0,T] ~\middle|~ \zeta(\sigma+it,\alpha) \in A \right\}, 
\quad
A \in \mathcal{B}(\mathbb{C})
\end{gather}
converges weakly as $T \to\infty$ to the law of the random variable $\zeta(\sigma,\mathbb{X}_\alpha)$. 
This is already achieved for the case in which $\alpha$ is a transcendental number. 
In this case, we choose $\mathbb{X}_\alpha(n)$ as independent random variables uniformly distributed on the unit circle of $\mathbb{C}$. 
Then we obtain the desired limit theorem for $\zeta(s,\alpha)$ as a consequence of \cite[Theorem 29]{JessenWintner1935}. 
Here, we remark that it is necessary for the proof to use the fact that the real numbers $\log(n+\alpha)$ are linearly independent over $\mathbb{Q}$ if $\alpha$ is transcendental. 
The linear independence of $\log(n+\alpha)$ is not necessarily valid for the case where $\alpha$ is algebraic. 
For example, Dubickas \cite{Dubickas1998} proved that the equation
\begin{gather*}
(x_1+\alpha)(x_2+\alpha)(x_3+\alpha)
= (u+\alpha)(v+\alpha)
\end{gather*}
has a solution in positive integers $x_1,x_2,x_3,u,v \geq n_0$ for any $n_0 \in \mathbb{Z}_{>0}$ if $\alpha$ is an algebraic integer of degree 2, which shows the $\mathbb{Q}$-linear dependence of the real numbers $\log(n+\alpha)$. 
Therefore we must change the choice of $\mathbb{X}_\alpha(n)$ to obtain the result for an algebraic parameter $\alpha$. 
The first main result provides suitable random variables $\mathbb{X}_\alpha(n)$. 

\begin{theorem}\label{thm:1.1}
Let $\alpha$ be an algebraic number satisfying $0<\alpha \leq1$. 
Then there exist random variables $\mathbb{X}_\alpha(n)$ for $n \geq0$ which are distributed on the unit circle, and 
\begin{gather}\label{eq:10010220}
\mathbf{E}\left[\mathbb{X}_\alpha(n_1)^{e_1} \cdots \mathbb{X}_\alpha(n_k)^{e_k}\right]
=
\begin{cases}
1
& \text{if $(n_1+\alpha)^{e_1}\cdots(n_k+\alpha)^{e_k}=1$},
\\
0
& \text{otherwise}
\end{cases}
\end{gather}
is satisfied for any integers $n_1, \ldots, n_k \geq0$ and $e_1, \ldots, e_k \in \mathbb{Z}$. 
\end{theorem}

The precise construction of $\mathbb{X}_\alpha(n)$ is explained in Section \ref{sec:2}. 
Note that we have $\overline{\mathbb{X}_\alpha(n)}=\mathbb{X}_\alpha(n)^{-1}$ since $\mathbb{X}_\alpha(n)$ are distributed on the unit circle. 
Therefore the random variables $\mathbb{X}_\alpha(n)$ of Theorem \ref{thm:1.1} are orthonormal in the sense that the following condition is satisfied: 
\begin{gather*}
\mathbf{E}[ \mathbb{X}_\alpha(m) \overline{\mathbb{X}_\alpha(n)}]
=
\begin{cases}
1
& \text{if $m=n$}, 
\\
0
& \text{otherwise}. 
\end{cases}
\end{gather*}
Then we apply the Menshov--Rademacher theorem \cite[Theorem B.10.5]{Kowalski2021} to see that \eqref{eq:10010000} is convergent for $\sigma>1/2$ almost surely due to 
\begin{gather*}
\sum_{n=0}^{\infty} \frac{(\log{n})^2}{(n+\alpha)^{2\sigma}}
< \infty.
\end{gather*}
With the above choice of $\mathbb{X}_\alpha(n)$, we prove the following limit theorem of $\zeta(s,\alpha)$ for any algebraic irrational number $\alpha$. 

\begin{theorem}\label{thm:1.2}
Let $\alpha$ be an algebraic number satisfying $0<\alpha \leq1$, and let $\mathbb{X}_\alpha(n)$ be the random variables as in Theorem \ref{thm:1.1}. 
Then, for any fixed real number $\sigma>1/2$, the probability measure $P_{\sigma,\alpha,T}$ defined as \eqref{eq:10010305} converges weakly as $T \to\infty$ to the law of $\zeta(\sigma,\mathbb{X}_\alpha)$. 
\end{theorem}

\begin{remark}\label{rem:1.4}
The first version of the limit theorem for $\zeta(s,\alpha)$ containing the case of algebraic irrational parameters was established by Laurin\v{c}ikas \cite{Laurincikas1994}. 
Let $\sigma>1/2$ be a fixed real number. 
Then he proved that the probability measure $P_{\sigma,\alpha,T}$ defined as \eqref{eq:10010305} converges weakly to some probability measure for any $0<\alpha \leq1$. 
However, this result did not present the limit measure explicitly. 
Laurin\v{c}ikas--Steuding \cite{LaurincikasSteuding2005} attempted to present the limit measure in terms of random variables for an algebraic irrational number $\alpha$. 
Let $L(\alpha)=\{\log(n+\alpha) \mid n \in \mathbb{Z}_{\geq0}\}$. 
Then we fix a maximal $\mathbb{Q}$-linearly independent subset of $L(\alpha)$ and denote it by $I(\alpha)$. 
Put
\begin{align*}
\mathcal{M}(\alpha)
&= \left\{m \in \mathbb{Z}_{\geq0} ~\middle|~ \log(m+\alpha) \in I(\alpha) \right\}, \\
\mathcal{N}(\alpha)
&= \left\{n \in \mathbb{Z}_{\geq0} ~\middle|~ \log(n+\alpha) \notin I(\alpha) \right\}. 
\end{align*}
Then we define random variables $\mathbb{X}_\alpha(m)$ for $m \in \mathcal{M}(\alpha)$ and $\mathbb{X}_\alpha(n)$ for $n \in \mathcal{N}(\alpha)$ as follows. 
Let $\mathbb{X}_\alpha(m)$ be independent random variables for $m \in \mathcal{M}(\alpha)$ which are uniformly distributed on the unit circle. 
For $n \in \mathcal{N}(\alpha)$, there exist elements $m_j \in \mathcal{M}(\alpha)$ and $r_j \in \mathbb{Q}$ for $j=1,\ldots,k$ such that
\begin{gather}\label{eq:10021516}
\log(n+\alpha)
= r_1 \log(m_1+\alpha) + \cdots + r_k \log(m_k+\alpha)
\end{gather}
since $\mathcal{M}(\alpha) \cup \{\log(n+\alpha)\}$ is linearly dependent over $\mathbb{Q}$. 
This implies that
\begin{gather*}
n+\alpha
= (m_1+\alpha)^{r_1} \cdots (m_k+\alpha)^{r_k}, 
\end{gather*}
and according to this identity, we define $\mathbb{X}_\alpha(n)$ for $n \in \mathcal{N}(\alpha)$ as 
\begin{gather}\label{eq:10021534}
\mathbb{X}_\alpha(n)
= \mathbb{X}_\alpha(m_1)^{r_1} \cdots \mathbb{X}_\alpha(m_k)^{r_k} 
\end{gather}
with the principal values of the roots. 
Laurin\v{c}ikas--Steuding \cite[Theorem 1]{LaurincikasSteuding2005} claimed that the probability measure $P_{\sigma,\alpha,T}$ converges weakly as $T \to\infty$ to the law of the random variable $\zeta(\sigma,\mathbb{X}_\alpha)$ defined as \eqref{eq:10010000}. 
However, it appears that the proof has a serious gap. 
They actually used the formula
\begin{align*}
(n+\alpha)^{it}
&= \big((m_1+\alpha)^{r_1} \cdots (m_k+\alpha)^{r_k}\big)^{it} \\ 
&= ((m_1+\alpha)^{it})^{r_1} \cdots ((m_k+\alpha)^{it})^{r_k} 
\end{align*}
for all $t \in \mathbb{R}$; see \cite[p.\,423, l.\,2]{LaurincikasSteuding2005}. 
Remark that the second equality fails when the rational numbers $r_j$ are not integers since they used the principal values of the roots, e.g.\ for any $\theta \in (\pi,2\pi)$ we have $\theta-2\pi \in (-\pi,\pi)$ and
\begin{gather*}
(e^{i \theta})^{1/2}
= (e^{i(\theta-2\pi)})^{1/2}
= e^{i (\theta-2\pi)/2}
= -(e^{1/2})^{i \theta}
\neq (e^{1/2})^{i \theta}. 
\end{gather*}
If one can confirm that all rational numbers $r_j$ in equation \eqref{eq:10021516} are integers for any $n \in \mathcal{N}(\alpha)$, then such a trouble does not occur and the proof of \cite{LaurincikasSteuding2005} will make sense. 
However, it seems quite difficult to show the integrality for $r_j$. 
The present paper avoids using the integrality for $r_j$, and we present another construction of the random variables $\mathbb{X}_\alpha(n)$. 
Then Theorem \ref{thm:1.2} is the first result on the limit theorem for $\zeta(s, \alpha)$ with algebraic parameter $\alpha$ whose limit measure is explicitly presented in terms of random variables. 
\end{remark}

Let $\sigma$ be a fixed real number with $1/2<\sigma \leq1$. 
Bohr--Courant \cite{BohrCourant1914} proved that the set $\{\zeta(\sigma+it) \mid t \in \mathbb{R} \}$ is dense in $\mathbb{C}$. 
Furthermore, this denseness result was generalized to the Hurwitz zeta-function $\zeta(s,\alpha)$ by Gonek \cite{Gonek1979} if $\alpha$ is rational or transcendental. 
More precisely, it holds that
\begin{gather}\label{eq:10010344}
\liminf_{T \to\infty} 
\frac{1}{T} \meas \left\{t \in [0,T] ~\middle|~ |\zeta(\sigma+it,\alpha)-z_0|<\epsilon \right\}
> 0
\end{gather}
for any $z_0 \in \mathbb{C}$ and $\epsilon>0$ for such $\alpha$. 
Gonek derived his result as a consequence of the so-called universality theorem for $\zeta(s,\alpha)$. 
He further conjectured in \cite{Gonek1979} that the universality theorem remains true if $\alpha$ is an algebraic irrational number. 
Then we also believe that \eqref{eq:10010344} is true for any algebraic irrational $\alpha$, but it is still open. 
Some evidence to this conjecture was recently presented by Lee--Mishou \cite{LeeMishou2020} and Sourmelidis--Steuding \cite{SourmelidisSteuding2022}. 
In the following, we denote by $\mathcal{A}$ the set of all algebraic irrational numbers $\alpha$ satisfying $0<\alpha<1$. 
Then we prove the following result which also supports the truth of \eqref{eq:10010344} for $\alpha \in \mathcal{A}$. 

\begin{theorem}\label{thm:1.3}
Let $\mathfrak{f}: \mathcal{A} \to \mathbb{R}_{>0}$ be the function defined later as in \eqref{eq:02232117}. 
Then the set $\mathcal{A}_d= \{\alpha \in \mathcal{A} \mid \mathfrak{f}(\alpha) \leq d \}$ contains infinitely many algebraic irrational numbers for $d \geq5$ and satisfies the following property. 
Let $\sigma$ be a fixed real number with $1/2<\sigma<1$. 
Then, for any $d \geq5$, $z_0 \in \mathbb{C}$, and $\epsilon>0$, there exists a finite subset $\mathcal{E}_{d}=\mathcal{E}_{d}(\sigma,z_0,\epsilon) \subset \mathcal{A}_{d}$ such that \eqref{eq:10010344} holds for any $\alpha \in \mathcal{A}_{d} \setminus \mathcal{E}_{d}$. 
\end{theorem}

Unfortunately, the exceptional subset $\mathcal{E}_{d}$ depends on $z_0$ and $\epsilon$. 
This prevents us from proving completely that the set $\{\zeta(\sigma+it,\alpha) \mid t \in \mathbb{R} \}$ is dense in $\mathbb{C}$. 
One can just deduce from Theorem \ref{thm:1.3} that $\{\zeta(\sigma+it, \alpha) \mid t \in \mathbb{R},~ \alpha \in \mathcal{A}_{d} \}$ is dense for any $d \geq5$. 
This is far from expected, but it makes a progress toward Gonek's conjecture by the probabilistic approach. 

\begin{remark}\label{rem:1.5}
Theorem \ref{thm:1.3} is similar to the recent result by Sourmelidis--Steuding \cite{SourmelidisSteuding2022}. 
The advantage of their result is that it is an effective result, which means that one can effectively find a real number $T>0$ such that $|\zeta(\sigma+it,\alpha)-z_0|<\epsilon$ holds with some $t \in [T,2T]$. 
Furthermore, they proved a weak form of the universality theorem for $\zeta(s,\alpha)$ for an algebraic irrational parameter $\alpha$ as a consequence of joint denseness results for $\zeta(s,\alpha)$ and its derivatives. 
As remarked in \cite[Section 1]{SourmelidisSteuding2022}, their results have meaning only when $\sigma>1-\xi$ with $\xi \approx 0.00186$. 
Theorem \ref{thm:1.3} of the present paper is not an effective result, but it has an advantage that we do not need such restriction for $\sigma$. 
\end{remark}

\subsection*{Organization of the paper}
This paper consists of five sections. 
\begin{itemize}
\item
Theorem \ref{thm:1.1} is proved in Section \ref{sec:2}. 
We construct certain random variables $\mathbb{X}_\alpha(n)$ such that condition \eqref{eq:10010220} is satisfied. 
Then we prove several properties of $\mathbb{X}_\alpha(n)$, which are used in the proof of Theorems \ref{thm:1.2} and \ref{thm:1.3}. 
\item
Theorem \ref{thm:1.2} is proved in Section \ref{sec:3}. 
A key step in the proof is to associate the complex numbers $(n_1+\alpha)^{-it}, \ldots, (n_k+\alpha)^{-it}$ for $t \in \mathbb{R}$ with the random variables $\mathbb{X}_\alpha(n_1), \ldots, \mathbb{X}_\alpha(n_k)$ by applying \eqref{eq:10010220}. 
\item
The remaining sections are devoted to the proof of Theorem \ref{thm:1.3}. 
The purpose of Section \ref{sec:4} is to show a variant of the Cassels lemma in \cite{Cassels1961}, which asserts that at least 51 percent of elements in $L(\alpha)=\{\log(n+\alpha) \mid n \in \mathbb{Z}_{\geq0}\}$ are linearly independent over $\mathbb{Q}$ in the sense of density. 
\item
Finally, the proof of Theorem \ref{thm:1.3} is completed in Section \ref{sec:6}. 
To show \eqref{eq:10010344}, we apply the limit theorem proved in Section \ref{sec:3} and the Cassels
lemma proved in Section \ref{sec:4}. 
The method is partially similar to that of Sourmelidis--Steuding \cite{SourmelidisSteuding2022}. 
\end{itemize}

\subsection*{Acknowledgment}
The author would like to thank Hidehiko Mishou for helpful comments on the method of Cassels \cite{Cassels1961} and Takashi Nakamura for discussion about the paper \cite{LaurincikasSteuding2005}. 
He is also grateful to the referee for many valuable suggestions and comments. 
The work of this paper was supported by JSPS Grant-in-Aid for Early-Career Scientists (Grant Number 24K16906).

\section{Certain random variables}\label{sec:2}
Throughout this paper, a random variable means a measurable map $\mathcal{X}: \Omega \to \mathbb{C}$ from a probability space $(\Omega, \mathcal{F}, \mathbf{P})$ to the Borel measurable space $(\mathbb{C}, \mathcal{B}(\mathbb{C}))$. 
Denote the expectation of a random variable $\mathcal{X}$ by 
\begin{gather*}
\mathbf{E}[\mathcal{X}]
= \int_{\Omega} \mathcal{X}(\omega) \,\mathbf{P}(d \omega)
\end{gather*}
as usual. 
Furthermore, the law of a random variable $\mathcal{X}$ is the measure on $(\mathbb{C}, \mathcal{B}(\mathbb{C}))$ defined as $\mathbf{P}\left(\mathcal{X} \in A\right)=\mathbf{P}\left(\{\omega \in \Omega \mid \mathcal{X}(\omega) \in A \}\right)$ for $A \in \mathcal{B}(\mathbb{C})$. 
Let $\Lambda$ be a countable set. 
Then it is well-known that there exist independent random variables $\mathcal{X}(\lambda)$ indexed by $\lambda \in \Lambda$ which are uniformly distributed on the unit circle on $\mathbb{C}$, i.e. the law of every $\mathcal{X}(\lambda)$ satisfies 
\begin{gather*}
\mathbf{P}\left(\mathcal{X}(\lambda) \in A(s,t)\right)
= \frac{t-s}{2\pi} 
\end{gather*}
for any $0<t-s \leq 2\pi$, where $A(s,t)$ is the arc of the unit circle such that
\begin{gather}\label{eq:10121028}
A(s,t)
= \{e^{i \theta} \mid s<\theta<t \}.
\end{gather}
See \cite[Section 5.1]{Laurincikas1996a} for such an example of random variables when $\Lambda$ is the set of prime numbers. 

Let $\alpha$ be an algebraic number satisfying $0<\alpha \leq1$. 
Then we define random variables $\mathbb{X}_\alpha(n)$ indexed by integers $n \geq0$ as follows. 
Let $K=\mathbb{Q}(\alpha)$ and denote by $\mathcal{O}_K$ the ring of integers of $K$. 
For any $y \in K$ with $y>0$, the fractional principal ideal $(y)$ has the decomposition
\begin{gather}\label{eq:10031324}
(y)
= \mathfrak{p}_1^{a_1} \cdots \mathfrak{p}_k^{a_k}, 
\end{gather}
where $\mathfrak{p}_j$ are prime ideals of $K$, and $a_j$ are rational integers uniquely determined by $y$. 
Let $h$ be the class number of $K$. 
Since $\mathfrak{a}^h$ is a principal ideal for any ideal $\mathfrak{a}$ of $K$, we know that the set
\begin{gather*}
S_\mathfrak{p}
= \{\varpi_\mathfrak{p} \in \mathcal{O}_K \mid \mathfrak{p}^h=(\varpi_\mathfrak{p}) \}
\end{gather*}
is non-empty for any prime ideal $\mathfrak{p}$. 
Then we take an element $\varpi = (\varpi_\mathfrak{p})_{\mathfrak{p}}$ in the set $\prod_{\mathfrak{p}} S_\mathfrak{p}$, where $\mathfrak{p}$ runs through all prime ideals of $K$. 
By \eqref{eq:10031324}, we obtain 
\begin{gather*}
(y)^h 
= \mathfrak{p}_1^{h a_1} \cdots \mathfrak{p}_k^{h a_k}
= (\varpi_{\mathfrak{p}_1}^{a_1} \cdots \varpi_{\mathfrak{p}_k}^{a_k}),
\end{gather*}
which yields the formula
\begin{gather*}
y^h
= \varpi_{\mathfrak{p}_1}^{a_1} \cdots \varpi_{\mathfrak{p}_k}^{a_k} \cdot u
\end{gather*}
for some $u \in \mathcal{O}_K^\times$. 
Here, the unit $u$ is uniquely determined only by $y$ if the choice of $\varpi = (\varpi_\mathfrak{p})_{\mathfrak{p}}$ is fixed. 
Furthermore, we can choose $\varpi$ so that every $\varpi_\mathfrak{p}$ is positive. 
Then we see that $u$ is positive due to $y^h>0$.  
Let $\mathcal{U}=(u_1, \ldots, u_d)$ be a fundamental system of units of $\mathcal{O}_K$. 
We can also choose $\mathcal{U}$ so that every $u_j$ is positive. 
With the above choices of $\varpi$ and $\mathcal{U}$, we obtain 
\begin{gather}\label{eq:10031358}
y^h
= \varpi_{\mathfrak{p}_1}^{a_1} \cdots \varpi_{\mathfrak{p}_k}^{a_k} u_1^{b_1} \cdots u_d^{b_d}, 
\end{gather}
where $b_j$ are rational integers uniquely determined by $y$. 
Let $\Lambda$ denote 
\begin{gather}\label{eq:10040341}
\Lambda
= \{\varpi_\mathfrak{p} \mid \text{$\mathfrak{p}$ is a prime ideal of $K$} \}
\cup \{u_1, u_2, \ldots, u_d\}. 
\end{gather}
Using the expression of $y^h$ as in \eqref{eq:10031358}, we define $\ord(y,\lambda)$ for $\lambda \in \Lambda$ as 
\begin{gather*}
\ord(y,\lambda)
=
\begin{cases}
a_j
& \text{if $\lambda=\varpi_{\mathfrak{p}_j}$ for some $\varpi_{\mathfrak{p}_j}$ in \eqref{eq:10031358}}, 
\\
b_j
& \text{if $\lambda=u_j$ for some $u_j$ in \eqref{eq:10031358}}, 
\\
0
& \text{otherwise}. 
\end{cases}
\end{gather*}
By the uniqueness of $a_j$ and $b_j$, we have the formula
\begin{gather}\label{eq:10031503}
\ord(y_1y_2,\lambda)
= \ord(y_1,\lambda)+\ord(y_2,\lambda)
\end{gather}
for any $y_1,y_2 \in K$ with $y_1,y_2>0$ and $\lambda \in \Lambda$. 
Let $\mathcal{X}(\lambda)$ be independent random variables for $\lambda \in \Lambda$ uniformly distributed on the unit circle. 
We finally define $\mathbb{X}_\alpha(n)$ for $n \geq0$ as 
\begin{gather}\label{eq:10031449}
\mathbb{X}_\alpha(n)
= \prod_{\lambda \in \Lambda} \mathcal{X}(\lambda)^{\ord(n+\alpha, \lambda)}. 
\end{gather}
In the following, we always denote by $\mathbb{X}_\alpha(n)$ the random variables defined in the above way. 
First, we show that they fulfill the desired condition of Theorem \ref{thm:1.1}. 

\begin{proof}[Proof of Theorem \ref{thm:1.1}]
By definition, the random variables $\mathbb{X}_\alpha(n)$ are distributed on the unit circle. 
Thus it remains to show below that condition \eqref{eq:10010220} is satisfied for any integers $n_1, \ldots, n_k \geq0$ and $e_1, \ldots, e_k \in \mathbb{Z}$. 
Put
\begin{gather*}
y
= (n_1+\alpha)^{e_1}\cdots(n_k+\alpha)^{e_k}. 
\end{gather*}
Then we have $y \in K$ and $y>0$. 
Using \eqref{eq:10031503}, we calculate $\ord(y,\lambda)$ for $\lambda \in \Lambda$ as 
\begin{gather*}
\ord(y,\lambda)
= e_1 \ord(n_1+\alpha,\lambda) +\cdots+ e_k \ord(n_k+\alpha,\lambda). 
\end{gather*}
Therefore we obtain
\begin{align*}
\mathbb{X}_\alpha(n_1)^{e_1} \cdots \mathbb{X}_\alpha(n_k)^{e_k}
&= \prod_{\lambda \in \Lambda} 
\mathcal{X}(\lambda)^{e_1 \ord(n_1+\alpha, \lambda) +\cdots+ e_k \ord(n_k+\alpha, \lambda)} \\
&= \prod_{\lambda \in \Lambda} 
\mathcal{X}(\lambda)^{\ord(y,\lambda)}
\end{align*}
by the definition of $\mathbb{X}_\alpha(n_j)$. 
Since the random variables $\mathcal{X}(\lambda)$ are independent, the expectation of the above is 
\begin{gather*}
\mathbf{E}\left[\mathbb{X}_\alpha(n_1)^{e_1} \cdots \mathbb{X}_\alpha(n_k)^{e_k}\right]
= \prod_{\lambda \in \Lambda} \mathbf{E}[\mathcal{X}(\lambda)^{\ord(y,\lambda)}]. 
\end{gather*}
Note that $\mathbf{E}[\mathcal{X}(\lambda)^{\ord(y,\lambda)}]$ equals to 1 if $\ord(y,\lambda)=0$ and to 0 otherwise since $\mathcal{X}(\lambda)$ is uniformly distributed on the unit circle. 
Hence we arrive at
\begin{gather}\label{eq:10031523}
\mathbf{E}\left[\mathbb{X}_\alpha(n_1)^{e_1} \cdots \mathbb{X}_\alpha(n_k)^{e_k}\right]
=
\begin{cases}
1
& \text{if $\ord(y,\lambda)=0$ for any $\lambda \in \Lambda$}, 
\\
0
& \text{otherwise}. 
\end{cases}
\end{gather}
The condition that $\ord(y,\lambda)=0$ for any $\lambda \in \Lambda$ is equivalent to $y^h=1$ by the expression of $y^h$ as in \eqref{eq:10031358}. 
Furthermore, $y^h=1$ if and only if $y=1$ due to $y>0$.  
Hence we see that \eqref{eq:10031523} is nothing but \eqref{eq:10010220}. 
\end{proof}

For a random variable $\mathcal{X}$, the characteristic function of $\mathcal{X}$ is defined as 
\begin{gather*}
g(w; \mathcal{X})
= \mathbf{E}\left[\psi_w(\mathcal{X})\right]
\end{gather*}
for $w \in \mathbb{C}$, where $\psi_w(z)=\exp(i \RE(z \overline{w}))$ is an additive character of $\mathbb{C} \simeq \mathbb{R}^2$. 
Then the law of $\mathcal{X}$ is uniquely determined by $g(w; \mathcal{X})$; see \cite[Section 29]{Billingsley1995}. 
For example, a random variable $\mathcal{X}$ is uniformly distributed on the unit circle if and only if the characteristic function is represented as
\begin{gather*}
g(w;\mathcal{X})
= \sum_{m=0}^{\infty} \frac{(-1)^m}{(m!)^2} \Big(\frac{|w|}{2}\Big)^{2m} 
= J_0(|w|), 
\end{gather*}
where $J_0(z)$ is the Bessel function of the first kind of order zero. 
In the remainder of this section, we prove several properties of $\mathbb{X}_\alpha(n)$. 

\begin{lemma}\label{lem:2.1}
Let $\alpha$ be an algebraic number satisfying $0<\alpha<1$. 
Then the random variables $\mathbb{X}_\alpha(n)$ are uniformly distributed on the unit circle. 
\end{lemma}

\begin{proof}
Since the additive character $\psi_w(z)$ is represented as
\begin{gather}\label{eq:10040314}
\psi_w(z)
= \exp\Big(\frac{i}{2}(z \overline{w}+\overline{z} w)\Big)
= \sum_{\mu=0}^{\infty} \sum_{\nu=0}^{\infty} 
\frac{(\frac{i}{2})^{\mu+\nu}}{\mu! \nu!}  
z^\mu \overline{z}^\nu \overline{w}^\mu w^\nu
\end{gather}
for any $z,w \in \mathbb{C}$, we calculate the characteristic function of $\mathbb{X}_\alpha(n)$ as
\begin{gather}\label{eq:10031708}
g(w; \mathbb{X}_\alpha(n))
= \sum_{\mu=0}^{\infty} \sum_{\nu=0}^{\infty} 
\frac{(\frac{i}{2})^{\mu+\nu}}{\mu! \nu!} 
\mathbf{E}[\mathbb{X}_\alpha(n)^\mu \overline{\mathbb{X}_\alpha(n)}^\nu]\,
\overline{w}^\mu w^\nu. 
\end{gather}
Applying \eqref{eq:10010220} with the identity $\overline{\mathbb{X}_\alpha(n)}=\mathbb{X}_\alpha(n)^{-1}$, we have 
\begin{gather*}
\mathbf{E}[\mathbb{X}_\alpha(n)^\mu \overline{\mathbb{X}_\alpha(n)}^\nu] 
= 
\begin{cases}
1
& \text{if $(n+\alpha)^\mu(n+\alpha)^{-\nu}=1$},
\\
0
& \text{otherwise}. 
\end{cases}
\end{gather*}
Note that $(n+\alpha)^\mu (n+\alpha)^{-\nu}=1$ if and only if $\mu=\nu$ since we have $n+\alpha \neq 1$ for $0<\alpha<1$. 
Therefore off-diagonal terms in \eqref{eq:10031708} disappear, and we obtain
\begin{gather}\label{eq:10040057}
g(w; \mathbb{X}_\alpha(n))
= \sum_{\mu=0}^{\infty} \frac{(-1)^\mu}{(\mu!)^2} \Big(\frac{|w|}{2}\Big)^{2\mu} 
= J_0(|w|). 
\end{gather}
This shows that $\mathbb{X}_\alpha(n)$ is uniformly distributed on the unit circle. 
\end{proof}

\begin{remark}\label{rem:2.2}
In the case of $\alpha=1$, we have $\mathbb{X}_\alpha(0) \equiv 1$ by definition. 
It is obviously not uniformly distributed on the unit circle. 
\end{remark}

For random variables $\mathcal{X}_1, \ldots, \mathcal{X}_k$, the joint characteristic function is defined as 
\begin{gather*}
g(w; \mathcal{X}_1, \ldots, \mathcal{X}_k)
= \mathbf{E}\left[\psi_{w_1}(\mathcal{X}_1) \cdots \psi_{w_k}(\mathcal{X}_k)\right]
\end{gather*}
for $w=(w_1, \ldots, w_k) \in \mathbb{C}^k$. 
By using the inversion formula \cite[Section 29]{Billingsley1995}, we find that $\mathcal{X}_1, \ldots, \mathcal{X}_k$ are independent if and only if the equality
\begin{gather*}
g(w; \mathcal{X}_1, \ldots, \mathcal{X}_k)
= g(w_1; \mathcal{X}_1) \cdots g(w_k; \mathcal{X}_k)
\end{gather*}
holds for any $w=(w_1, \ldots, w_k) \in \mathbb{C}^k$. 
By this fact, we prove the following result. 

\begin{lemma}\label{lem:2.4}
Let $\alpha$ be an algebraic number satisfying $0<\alpha<1$, and let
\begin{align*}
\mathcal{L}(N)
&= \{n \in \mathbb{Z} \mid N<n \leq N \log{N} \}, \\
\mathcal{M}(N)
&= \{n \in \mathbb{Z} \mid 0 \leq n \leq N \log{N} \}
\end{align*}
for $N \geq3$. 
We define $\mathcal{K}_\alpha(N)$ as the subset of $\mathcal{L}(N)$ consisting of all integers $n$ such that $(n+\alpha) \mathfrak{a}$ is divisible by a prime ideal $\mathfrak{p}$ not dividing $(m+\alpha) \mathfrak{a}$ for any $m \in \mathcal{M}(N)$ with $m \neq n$, where $\mathfrak{a}$ denotes the ideal denominator of $\alpha$. 
Put
\begin{gather*}
\mathcal{K}_\alpha(N)
= \left\{ n_1, \ldots, n_k \right\}
\quad\text{and}\quad
\mathcal{M}(N) \setminus \mathcal{K}_\alpha(N)
= \left\{ n_{k+1}, \ldots, n_\ell \right\}. 
\end{gather*}
Then the random variables 
\begin{gather*}
\mathbb{X}_\alpha(n_1), \ldots, \mathbb{X}_\alpha(n_k), \,
c_{k+1} \mathbb{X}_\alpha(n_{k+1})+ \cdots +c_\ell \mathbb{X}_\alpha(n_\ell)
\end{gather*}
are independent for any complex numbers $c_{k+1}, \ldots, c_\ell$. 
\end{lemma}

\begin{proof}
Define $\mathcal{Y}=c_{k+1} \mathbb{X}_\alpha(n_{k+1})+ \cdots +c_\ell \mathbb{X}_\alpha(n_\ell)$. 
Then we have 
\begin{align*}
&g(w; \mathbb{X}_\alpha(n_1), \ldots, \mathbb{X}_\alpha(n_k), \mathcal{Y}) \\
&= \mathbf{E}\left[\psi_{w_1}(\mathbb{X}_\alpha(n_1)) \cdots \psi_{w_k}(\mathbb{X}_\alpha(n_k)) \,
\psi_{w_{k+1}}(\mathcal{Y}) \right] \\
&= \mathbf{E}\left[\psi_{w_1}(\mathbb{X}_\alpha(n_1)) \cdots \psi_{w_k}(\mathbb{X}_\alpha(n_k)) \,
\psi_{w_{k+1}}(c_{k+1} \mathbb{X}_\alpha(n_{k+1})) \cdots \psi_{w_{k+1}}(c_\ell \mathbb{X}_\alpha(n_\ell)) \right]
\end{align*}
for $w=(w_1, \ldots, w_{k+1}) \in \mathbb{C}^{k+1}$. 
Expanding $\psi_w(z)$ by using \eqref{eq:10040314}, we have
\begin{align}\label{eq:10040410}
&g(w; \mathbb{X}_\alpha(n_1), \ldots, \mathbb{X}_\alpha(n_k), \mathcal{Y}) \\
&= \sum_{\mu_1=0}^{\infty} \sum_{\nu_1=0}^{\infty} \cdots \sum_{\mu_\ell=0}^{\infty} \sum_{\nu_\ell=0}^{\infty}
\frac{(\frac{i}{2})^{\mu_1+\nu_1+\cdots+\mu_\ell+\nu_\ell}}{\mu_1! \nu_1! \cdots \mu_\ell! \nu_\ell!} 
G_{\mu_1,\nu_1, \ldots, \mu_\ell, \nu_\ell} \, 
c_{k+1}^{\mu_{k+1}} \overline{c_{k+1}}^{\nu_{k+1}} \cdots c_\ell^{\mu_\ell} \overline{c_\ell}^{\nu_\ell} 
\nonumber \\
&\hspace{40truemm}\times 
\overline{w_1}^{\mu_1} w_1^{\nu_1} \cdots \overline{w_k}^{\mu_k} w_k^{\nu_k}
\overline{w_{k+1}}^{\mu_{k+1}+ \cdots +\mu_\ell} w_{k+1}^{\nu_{k+1}+ \cdots +\nu_\ell}, 
\nonumber
\end{align}
where $G_{\mu_1,\nu_1, \ldots, \mu_\ell, \nu_\ell} = \mathbf{E}[\mathbb{X}_\alpha(n_1)^{\mu_1} \overline{\mathbb{X}_\alpha(n_1)}^{\nu_1} \cdots \mathbb{X}_\alpha(n_\ell)^{\mu_\ell} \overline{\mathbb{X}_\alpha(n_\ell)}^{\nu_\ell}]$. 
By the definition of $\mathcal{K}_\alpha(N)$, there exists a prime ideal $\mathfrak{p}_j$ for $j \in \{1, \ldots, k\}$ such that it is a prime divisor of $(n_j+\alpha) \mathfrak{a}$ but not dividing $(m+\alpha) \mathfrak{a}$ for any integer $m \in \mathcal{M}(N)$ with $m \neq n_j$. 
Hence, for $\lambda_j = \varpi_{\mathfrak{p}_j} \in \Lambda$ as in \eqref{eq:10040341}, we have $\ord(n_j+\alpha, \lambda_j)
>0$ but $\ord(m+\alpha, \lambda_j)=0$ for any $m \in \mathcal{M}(N)$ with $m \neq n_j$. 
Note that $\lambda_1,\ldots,\lambda_k$ are distinct. 
Since every $\mathbb{X}_\alpha(n_j)$ is represented as
\begin{gather*}
\mathbb{X}_\alpha(n_j) 
= \mathcal{X}(\lambda_j)^{\ord(n_j+\alpha, \lambda_j)} 
\prod_{\lambda \in \Lambda, \, \lambda \neq \lambda_j} \mathcal{X}(\lambda)^{\ord(n_j+\alpha, \lambda)} 
\end{gather*}
by \eqref{eq:10031449}, we obtain
\begin{align*}
&\mathbb{X}_\alpha(n_1)^{\mu_1} \overline{\mathbb{X}_\alpha(n_1)}^{\nu_1} \cdots \mathbb{X}_\alpha(n_\ell)^{\mu_\ell} \overline{\mathbb{X}_\alpha(n_\ell)}^{\nu_\ell} \\ 
&= \prod_{j=1}^{k} \mathcal{X}(\lambda_j)^{(\mu_j-\nu_j) \ord(n_j+\alpha, \lambda_j)} \\
&\qquad \times 
\prod_{\lambda \in \Lambda, \, \lambda \neq \lambda_1,\ldots,\lambda_k} 
\mathcal{X}(\lambda)^{(\mu_1-\nu_1) \ord(n_1+\alpha, \lambda)+ \cdots +(\mu_\ell-\nu_\ell) \ord(n_\ell+\alpha, \lambda)}. 
\end{align*}
By the independence of $\mathcal{X}(\lambda)$, it follows that
\begin{align*}
G_{\mu_1,\nu_1, \ldots, \mu_\ell, \nu_\ell}
&= \prod_{j=1}^{k} \mathbf{E}[\mathcal{X}(\lambda_j)^{(\mu_j-\nu_j) \ord(n_j+\alpha, \lambda_j)}] \\
&\quad \times 
\prod_{\lambda \in \Lambda, \, \lambda \neq \lambda_1,\ldots,\lambda_k} 
\mathbf{E}[\mathcal{X}(\lambda)^{(\mu_1-\nu_1) \ord(n_1+\alpha, \lambda)+ \cdots +(\mu_\ell-\nu_\ell) \ord(n_\ell+\alpha, \lambda)}]. 
\end{align*}
As a result, we see that $G_{\mu_1,\nu_1, \ldots, \mu_\ell, \nu_\ell}=0$ if $\mu_j \neq \nu_j$ for some $1 \leq j \leq k$. 
Hence the terms in \eqref{eq:10040410} with $\mu_j \neq \nu_j$ for some $1 \leq j \leq k$ disappear. 
Furthermore, we have
\begin{gather*}
G_{\mu_1,\nu_1, \ldots, \mu_\ell, \nu_\ell} 
= \mathbf{E}[\mathbb{X}_\alpha(n_{k+1})^{\mu_{k+1}} \overline{\mathbb{X}_\alpha(n_{k+1})}^{\nu_{k+1}} 
\cdots \mathbb{X}_\alpha(n_\ell)^{\mu_\ell} \overline{\mathbb{X}_\alpha(n_\ell)}^{\nu_\ell}]
\end{gather*}
if $\mu_j=\nu_j$ for all $1 \leq j \leq k$ by definition. 
Thus we deduce from \eqref{eq:10040410} that
\begin{align*}
&g(w; \mathbb{X}_\alpha(n_1), \ldots, \mathbb{X}_\alpha(n_k), \mathcal{Y}) \\
&= \sum_{\mu_1=0}^{\infty} \cdots \sum_{\mu_k=0}^{\infty} 
\sum_{\mu_{k+1}=0}^{\infty} \sum_{\nu_{k+1}=0}^{\infty} \cdots \sum_{\mu_\ell=0}^{\infty} \sum_{\nu_\ell=0}^{\infty}
\frac{(-1)^{\mu_1+\cdots+\mu_k}}{(\mu_1!)^2 \cdots (\mu_k!)^2} 
\frac{(\frac{i}{2})^{\mu_{k+1}+\nu_{k+1}+\cdots+\mu_\ell+\nu_\ell}}{\mu_{k+1}! \nu_{k+1}! \cdots \mu_\ell! \nu_\ell!} \\
&\qquad \times 
\mathbf{E}[\mathbb{X}_\alpha(n_{k+1})^{\mu_{k+1}} \overline{\mathbb{X}_\alpha(n_{k+1})}^{\nu_{k+1}} 
\cdots \mathbb{X}_\alpha(n_\ell)^{\mu_\ell} \overline{\mathbb{X}_\alpha(n_\ell)}^{\nu_\ell}] \\
&\qquad \times 
\Big(\frac{|w_1|}{2}\Big)^{2\mu_1} \cdots \Big(\frac{|w_k|}{2}\Big)^{2\mu_k}
c_{k+1}^{\mu_{k+1}} \overline{c_{k+1}}^{\nu_{k+1}} \cdots c_\ell^{\mu_\ell} \overline{c_\ell}^{\nu_\ell}
\overline{w_{k+1}}^{\mu_{k+1}+ \cdots +\mu_\ell} w_{k+1}^{\nu_{k+1}+ \cdots +\nu_\ell} \\
&= g(w_1; \mathbb{X}_\alpha(n_1)) \cdots g(w_k; \mathbb{X}_\alpha(n_k)) \, 
\mathfrak{g}(w_{k+1}) 
\end{align*}
by using \eqref{eq:10040057}, where we put
\begin{align*}
\mathfrak{g}(w_{k+1})
&= \sum_{\mu_{k+1}=0}^{\infty} \sum_{\nu_{k+1}=0}^{\infty} \cdots \sum_{\mu_\ell=0}^{\infty} \sum_{\nu_\ell=0}^{\infty}
\frac{(\frac{i}{2})^{\mu_{k+1}+\nu_{k+1}+\cdots+\mu_\ell+\nu_\ell}}{\mu_{k+1}! \nu_{k+1}! \cdots \mu_\ell! \nu_\ell!} \\
&\qquad \times
\mathbf{E}[\mathbb{X}_\alpha(n_{k+1})^{\mu_{k+1}} \overline{\mathbb{X}_\alpha(n_{k+1})}^{\nu_{k+1}} 
\cdots \mathbb{X}_\alpha(n_\ell)^{\mu_\ell} \overline{\mathbb{X}_\alpha(n_\ell)}^{\nu_\ell}] \\
&\qquad \times 
c_{k+1}^{\mu_{k+1}} \overline{c_{k+1}}^{\nu_{k+1}} \cdots c_\ell^{\mu_\ell} \overline{c_\ell}^{\nu_\ell}
\overline{w_{k+1}}^{\mu_{k+1}+ \cdots +\mu_\ell} w_{k+1}^{\nu_{k+1}+ \cdots +\nu_\ell}.
\end{align*}
Here, we can confirm the identity $\mathfrak{g}(w_{k+1})=g(w_{k+1}; \mathcal{Y})$ by using the expansion of $\psi_w(z)$ as in \eqref{eq:10040314}. 
Therefore the equality
\begin{gather*}
g(w; \mathbb{X}_\alpha(n_1), \ldots, \mathbb{X}_\alpha(n_k), \mathcal{Y})
= g(w_1; \mathbb{X}_\alpha(n_1)) \cdots g(w_k; \mathbb{X}_\alpha(n_k)) g(w_{k+1}; \mathcal{Y})
\end{gather*}
holds, and we obtain the desired result. 
\end{proof}

\section{Proof of the limit theorem}\label{sec:3}
Let $\sigma>1/2$ be a fixed real number, and let $\alpha$ be an algebraic number satisfying $0<\alpha \leq1$. 
For $T>0$, we define
\begin{gather}\label{eq:10051642}
g_T(w; \sigma, \alpha)
= \frac{1}{T} \int_{0}^{T} \psi_w(\zeta(\sigma+it, \alpha)) \,dt, 
\end{gather}
where $\psi_w(z)=\exp(i \RE(z \overline{w}))$ for $z,w \in \mathbb{C}$ as before. 
In this section, we show that the function $g_T(w; \sigma, \alpha)
$ converges to 
\begin{gather*}
g(w; \sigma, \mathbb{X}_\alpha)
= \mathbf{E}\left[\psi_w(\zeta(\sigma, \mathbb{X}_\alpha))\right]
\end{gather*}
as $T \to\infty$ uniformly in the region $|w| \leq R$ for any $R>0$, where $\zeta(\sigma,  \mathbb{X}_\alpha)$ is the random variable of \eqref{eq:10010000}. 
For this, we truncate the series of $\zeta(s,\alpha)$ and $\zeta(\sigma, \mathbb{X}_\alpha)$ as 
\begin{gather}\label{eq:10121134}
\zeta_N(s,\alpha)
= \sum_{n=0}^{N} \frac{1}{(n+\alpha)^s}
\quad\text{and}\quad
\zeta_N(\sigma, \mathbb{X}_\alpha)
= \sum_{n=0}^{N} \frac{\mathbb{X}_\alpha(n)}{(n+\alpha)^\sigma}
\end{gather}
with an integer $N \geq0$, and define
\begin{align}
g_{T,N}(w; \sigma, \alpha)
&= \frac{1}{T} \int_{0}^{T} \psi_w(\zeta_N(\sigma+it, \alpha)) \,dt, 
\label{eq:10051753} \\
g_N(w; \sigma, \mathbb{X}_\alpha)
&= \mathbf{E}\left[\psi_w(\zeta_N(\sigma, \mathbb{X}_\alpha))\right]
\nonumber
\end{align}
similarly to the above. 
Then we have the following propositions. 

\begin{proposition}\label{prop:3.1}
Let $\sigma>1/2$ be a fixed real number, and let $\alpha$ be an algebraic number satisfying $0<\alpha \leq1$. 
Take an integer $N \geq2$ arbitrary. 
For any $R>0$ and $\epsilon>0$, there exists a positive real number $T_0=T_0(\alpha,N,R,\epsilon)$ such that the inequality
\begin{gather*}
|g_{T,N}(w; \sigma, \alpha)-g_N(w; \sigma, \mathbb{X}_\alpha)|
< \epsilon
\end{gather*}
holds for all $T \geq T_0$ in the region $|w| \leq R$. 
\end{proposition}

\begin{proof}
Let $M$ be any positive integer. 
Recall that the asymptotic formula
\begin{gather*}
e^{i \theta}
= \sum_{m<M} \frac{i^m}{m!} \theta^m
+ O\left(\frac{1}{M!}|\theta|^M\right)
\end{gather*}
holds for all $\theta \in \mathbb{R}$ with an absolute implied constant. 
Then we obtain
\begin{gather*}
\psi_w(z)
= \sum_{\mu+\nu<M} \frac{(\frac{i}{2})^{\mu+\nu}}{\mu! \nu!} 
z^\mu \overline{z}^\nu \overline{w}^\mu w^\nu
+ O\left(\frac{1}{M!}|zw|^M\right)
\end{gather*}
for all $z,w \in \mathbb{C}$. 
Applying this, we calculate the right-hand side of \eqref{eq:10051753} as
\begin{align}
&g_{T,N}(w; \sigma, \alpha) 
\label{eq:10052032} \\
&= \sum_{\mu+\nu<M} \frac{(\frac{i}{2})^{\mu+\nu}}{\mu! \nu!} 
\left( \frac{1}{T} \int_{0}^{T} \zeta_N(\sigma+it, \alpha)^\mu \overline{\zeta_N(\sigma+it, \alpha)}^\nu \,dt \right)
\overline{w}^\mu w^\nu
+ E_1,
\nonumber
\end{align}
where the error term is evaluated as
\begin{gather*}
E_1
\ll \frac{R^M}{M!} \frac{1}{T} \int_{0}^{T} |\zeta_N(\sigma+it, \alpha)|^M \,dt
\end{gather*}
for $|w| \leq R$ with an absolute implied constant. 
By definition, we have 
\begin{gather*}
|\zeta_N(\sigma+it, \alpha)|
\leq \sum_{n=0}^{N} \frac{1}{(n+\alpha)^{1/2}}
\ll_\alpha \sqrt{N}
\end{gather*}
for any $\sigma>1/2$. 
Hence we obtain $E_1 \ll (c(\alpha) R \sqrt{N})^M/M! $ with a positive constant $c(\alpha)$ depending only on $\alpha$. 
As a result, there exists an integer $M=M(\alpha,N,R,\epsilon)$ such that $|E_1|<\epsilon/3$ holds for each $\epsilon>0$. 
We can choose $M$ so that $M \geq2$. 
Then, the main term of \eqref{eq:10052032} is evaluated as follows. 
We have 
\begin{align*}
&\frac{1}{T} \int_{0}^{T} \zeta_N(\sigma+it, \alpha)^\mu \overline{\zeta_N(\sigma+it, \alpha)}^\nu \,dt \\
&= \sum_{m_1=0}^{N} \cdots \sum_{m_\mu=0}^{N} \sum_{n_1=0}^{N} \cdots \sum_{n_\nu=0}^{N}
\frac{1}{(m_1+\alpha)^\sigma \cdots (m_\mu+\alpha)^\sigma} \\
&\qquad \times 
\frac{1}{(n_1+\alpha)^\sigma \cdots (n_\nu+\alpha)^\sigma}
\frac{1}{T} \int_{0}^{T} \left(\frac{(m_1+\alpha) \cdots (m_\mu+\alpha)}{(n_1+\alpha) \cdots (n_\nu+\alpha)}\right)^{-it} \,dt \\
&= \mathop{\sum_{m_1=0}^{N} \cdots \sum_{m_\mu=0}^{N} \sum_{n_1=0}^{N} \cdots \sum_{n_\nu=0}^{N}}
\limits_{(m_1+\alpha) \cdots (m_\mu+\alpha) = (n_1+\alpha) \cdots (n_\nu+\alpha)}
\frac{1}{(m_1+\alpha)^\sigma \cdots (m_\mu+\alpha)^\sigma} 
\frac{1}{(n_1+\alpha)^\sigma \cdots (n_\nu+\alpha)^\sigma} \\
&\qquad 
+\mathop{\sum_{m_1=0}^{N} \cdots \sum_{m_\mu=0}^{N} \sum_{n_1=0}^{N} \cdots \sum_{n_\nu=0}^{N}}
\limits_{(m_1+\alpha) \cdots (m_\mu+\alpha) \neq (n_1+\alpha) \cdots (n_\nu+\alpha)}
\frac{1}{(m_1+\alpha)^\sigma \cdots (m_\mu+\alpha)^\sigma} \\
&\qquad\qquad \times 
\frac{1}{(n_1+\alpha)^\sigma \cdots (n_\nu+\alpha)^\sigma}
\frac{1}{iT} \left\{1-\left(\frac{(m_1+\alpha) \cdots (m_\mu+\alpha)}{(n_1+\alpha) \cdots (n_\nu+\alpha)}\right)^{-iT}\right\}\\
&\qquad\qquad \times 
\left|\log{\frac{(m_1+\alpha) \cdots (m_\mu+\alpha)}{(n_1+\alpha) \cdots (n_\nu+\alpha)}}\right|^{-1} \\
&=S_1+S_2, 
\end{align*}
say. 
Recall that the random variables $\mathbb{X}_\alpha(n)$ satisfy \eqref{eq:10010220}. 
Then the first term $S_1$ is
\begin{align}\label{eq:10052256}
S_1
&= \sum_{m_1=0}^{N} \cdots \sum_{m_\mu=0}^{N} \sum_{n_1=0}^{N} \cdots \sum_{n_\nu=0}^{N}
\frac{1}{(m_1+\alpha)^\sigma \cdots (m_\mu+\alpha)^\sigma} \\
&\qquad \times 
\frac{1}{(n_1+\alpha)^\sigma \cdots (n_\nu+\alpha)^\sigma}
\mathbf{E}\left[\mathbb{X}_\alpha(m_1) \cdots \mathbb{X}_\alpha(m_\mu) 
\mathbb{X}_\alpha(n_1)^{-1} \cdots \mathbb{X}_\alpha(n_\nu)^{-1} \right] 
\nonumber \\
&= \mathbf{E}\left[\zeta_N(\sigma, \mathbb{X}_\alpha)^\mu \overline{\zeta_N(\sigma, \mathbb{X}_\alpha)}^\nu\right]. 
\nonumber
\end{align}
On the other hand, the second term $S_2$ is evaluated as follows. 
Put 
\begin{gather*}
P(x)
= (m_1+x) \cdots (m_\mu+x) - (n_1+x) \cdots (n_\nu+x), 
\end{gather*}
and denote its degree and height by $\deg(P)$ and $\mathrm{ht}(P)$, respectively. 
For $\mu+\nu<M$, $\deg(P)$ is at most $M$. 
For $m_1,\ldots,m_\mu, n_1,\ldots,n_\nu \leq N$, we also find that $\mathrm{ht}(P)$ is less than the height of $(N+x)^M$, which is at most $(2N)^M$. 
We also denote by $\deg(\alpha)$ and $\mathrm{ht}(\alpha)$ the degree and height of $\alpha$, respectively. 
Then we obtain 
\begin{align*}
|P(\alpha)|
&\geq (\deg(P)+1)^{1-\deg(\alpha)} (\deg(\alpha)+1)^{-\deg(P)/2}
\mathrm{ht}(P)^{1-\deg(\alpha)} \mathrm{ht}(\alpha)^{-\deg(P)} \\
&\geq (2N)^{-\omega(\alpha) M}
\end{align*}
by applying \cite[Theorem A.1]{Bugeaud2004}, where $\omega(\alpha)$ is a positive constant that depends only on $\alpha$. 
Hence we derive
\begin{align*}
&\left|\log \frac{(m_1+\alpha) \cdots (m_\mu+\alpha)}{(n_1+\alpha) \cdots (n_\nu+\alpha)}\right| \\
&\geq \frac{|(m_1+\alpha) \cdots (m_\mu+\alpha)-(n_1+\alpha) \cdots (n_\nu+\alpha)|}
{\max\{ (m_1+\alpha) \cdots (m_\mu+\alpha), (n_1+\alpha) \cdots (n_\nu+\alpha) \}} \\
&\geq (2N)^{-(\omega(\alpha)+1) M}. 
\end{align*}
Therefore we arrive at
\begin{align}\label{eq:10052257}
S_2
&\ll \left(\sum_{n=0}^{N} \frac{1}{(n+\alpha)^\sigma}\right)^{\mu+\nu} 
\frac{1}{T} (2N)^{(\omega(\alpha)+1) M} \\
&\leq \frac{1}{T} (d(\alpha) N)^{(\omega(\alpha)+1) M}
\nonumber
\end{align}
with a positive constant $d(\alpha)$ depending only on $\alpha$. 
Combining \eqref{eq:10052256} and \eqref{eq:10052257}, we deduce the formula
\begin{align*}
&\frac{1}{T} \int_{0}^{T} \zeta_N(\sigma+it, \alpha)^\mu \overline{\zeta_N(\sigma+it, \alpha)}^\nu \,dt \\
&= \mathbf{E}\left[\zeta_N(\sigma, \mathbb{X}_\alpha)^\mu \overline{\zeta_N(\sigma, \mathbb{X}_\alpha)}^\nu\right]
+ O\left(\frac{1}{T} (d(\alpha) N)^{(\omega(\alpha)+1) M}\right), 
\end{align*}
where the implied constant is absolute. 
Furthermore, it yields
\begin{align}
&\sum_{\mu+\nu<M} \frac{(\frac{i}{2})^{\mu+\nu}}{\mu! \nu!} 
\left( \frac{1}{T} \int_{0}^{T} \zeta_N(\sigma+it, \alpha)^\mu \overline{\zeta_N(\sigma+it, \alpha)}^\nu \,dt \right)
\overline{w}^\mu w^\nu
\label{eq:10052322} \\
&= \sum_{\mu+\nu<M} \frac{(\frac{i}{2})^{\mu+\nu}}{\mu! \nu!} 
\mathbf{E}\left[\zeta_N(\sigma, \mathbb{X}_\alpha)^\mu \overline{\zeta_N(\sigma, \mathbb{X}_\alpha)}^\nu\right]
\overline{w}^\mu w^\nu 
+E_2, 
\nonumber
\end{align}
where the error term is estimated as
\begin{align*}
E_2
&\ll \frac{1}{T} (d(\alpha) N)^{(\omega(\alpha)+1) M}
\sum_{\mu+\nu<M} \frac{1}{\mu! \nu!} 
\Big(\frac{|w|}{2}\Big)^{\mu+\nu} \\
&\ll \frac{1}{T} (d(\alpha) N)^{(\omega(\alpha)+1) M} \exp(R) 
\end{align*}
for $|w| \leq R$.
Thus, there exists a positive real number $T_0=T_0(\alpha,M,N,R,\epsilon)$ such that $|E_2|<\epsilon/3$ holds for all $T \geq T_0$. 
Finally, we obtain
\begin{align}
&g_N(w; \sigma, \mathbb{X}_\alpha) 
\label{eq:10052323} \\
&= \sum_{\mu+\nu<M} \frac{(\frac{i}{2})^{\mu+\nu}}{\mu! \nu!} 
\mathbf{E}\left[\zeta_N(\sigma, \mathbb{X}_\alpha)^\mu \overline{\zeta_N(\sigma, \mathbb{X}_\alpha)}^\nu\right]
\overline{w}^\mu w^\nu
+ E_3
\nonumber
\end{align}
similarly to \eqref{eq:10052032}, where
\begin{gather*}
E_3
\ll \frac{R^M}{M!} \mathbf{E}\left[|\zeta_N(\sigma, \mathbb{X}_\alpha)|^M\right]
\ll \frac{(c(\alpha) R \sqrt{N})^M}{M!}
\end{gather*}
with absolute implied constants.
Then we again obtain $|E_3|< \epsilon/3$ by recalling the choice of $M$. 
The desired result follows from the above formulas \eqref{eq:10052032}, \eqref{eq:10052322}, \eqref{eq:10052323} with the above error estimates. 
\end{proof}

\begin{proposition}\label{prop:3.2}
Let $\sigma>1/2$ be a fixed real number, and let $\alpha$ be an algebraic number satisfying $0<\alpha \leq1$. 
For any $R>0$ and $\epsilon>0$, there exists an integer $N_0=N_0(\sigma,R,\epsilon) \geq0$ such that the inequalities
\begin{align*}
\limsup_{T \to\infty} 
|g_T(w; \sigma, \alpha)
&-g_{T,N}(w; \sigma, \alpha)|
< \epsilon, \\
|g(w; \sigma, \mathbb{X}_\alpha)&-g_N(w; \sigma, \mathbb{X}_\alpha)|
< \epsilon
\end{align*}
hold for all $N \geq N_0$ in the region $|w| \leq R$. 
\end{proposition}

\begin{proof}
Since the inequality $|e^{i \theta_1}-e^{i \theta_2}| \leq |\theta_1-\theta_2|$ holds for all $\theta_1,\theta_2 \in \mathbb{R}$, we have 
\begin{gather}\label{eq:10060256}
|\psi_w(z_1)-\psi_w(z_2)|
\leq |z_1-z_2| \, |w|
\end{gather}
for all $z_1, z_2, w \in \mathbb{C}$. 
By \eqref{eq:10051642}, \eqref{eq:10051753}, and $|\psi_w(z)|\leq1$, we obtain
\begin{align}\label{eq:10060249}
&|g_T(w; \sigma, \alpha)
-g_{T,N}(w; \sigma, \alpha)| \\
&\leq \frac{R}{T} \int_{2\pi}^{T} |\zeta(\sigma+it, \alpha)-\zeta_N(\sigma+it, \alpha)| \,dt
+ \frac{4\pi}{T}
\nonumber
\end{align}
for $|w| \leq R$. 
We evaluate this integral as follows. 
If $\sigma \geq2$, then we have $\zeta(s,\alpha)=\zeta_N(s,\alpha)+O(N^{-1})$ by applying \eqref{eq:10060143}. 
It yields the estimate
\begin{gather}\label{eq:10060250}
\frac{1}{T} \int_{2\pi}^{T} |\zeta(\sigma+it, \alpha)-\zeta_N(\sigma+it, \alpha)| \,dt
\ll N^{-1}. 
\end{gather}
If $1/2<\sigma<2$, then we apply the formula \cite[Theorem 1 on p.\,78]{KaratsubaVoronin1992} 
\begin{gather*}
\zeta(s,\alpha)
= \sum_{0 \leq n \leq x} \frac{1}{(n+\alpha)^s}
+ \frac{x^{1-s}}{s-1}
+ O\left(x^{-\sigma}\right)
\end{gather*}
for $2\pi \leq |t| \leq \pi x$. 
Taking $x= T/\pi$, we obtain
\begin{align*}
&\frac{1}{T} \int_{2\pi}^{T} |\zeta(\sigma+it, \alpha)-\zeta_N(\sigma+it, \alpha)| \,dt \\
&\ll \left( \frac{1}{T} \int_{0}^{T} \left|\sum_{N<n \leq x} \frac{1}{(n+\alpha)^{\sigma+it}}\right|^2 \,dt \right)^{1/2}
+ \frac{1}{T} (\log{T}) x^{1-\sigma}
+ x^{-\sigma}
\end{align*}
due to the Cauchy--Schwarz inequality. 
Furthermore, it holds that 
\begin{gather*}
\int_{0}^{T} \left|\sum_{N<n \leq x} \frac{1}{(n+\alpha)^{\sigma+it}}\right|^2 \,dt
\ll (T+x) N^{1-2\sigma} 
\end{gather*}
by \cite[Corollary 2]{MontgomeryVaughan1974}, where the implied constant depends only on $\sigma$. 
Hence we have 
\begin{gather}\label{eq:10060251}
\frac{1}{T} \int_{2\pi}^{T} |\zeta(\sigma+it, \alpha)-\zeta_N(\sigma+it, \alpha)| \,dt
\ll_\sigma N^{1/2-\sigma} + T^{-\sigma} \log{T}
\end{gather}
in this case. 
Using \eqref{eq:10060250} for $\sigma \geq2$ and \eqref{eq:10060251} for $1/2<\sigma<2$, we obtain
\begin{gather*}
\limsup_{T \to\infty} 
|g_T(w; \sigma, \alpha)
-g_{T,N}(w; \sigma, \alpha)|
\ll_\sigma R\, N^{\max(-1,1/2-\sigma)}
\end{gather*}
by \eqref{eq:10060249}. 
This yields the first inequality of the result. 
Then, we prove the second inequality. 
By \eqref{eq:10060256} and the Cauchy--Schwarz inequality, we have 
\begin{align*}
|g(w; \sigma, \mathbb{X}_\alpha)-g_N(w; \sigma, \mathbb{X}_\alpha)|
&\leq R\, \mathbf{E}\left[|\zeta(\sigma,\mathbb{X}_\alpha)-\zeta_N(\sigma,\mathbb{X}_\alpha)|\right] \\
&\leq R\, \mathbf{E}\left[|\zeta(\sigma,\mathbb{X}_\alpha)-\zeta_N(\sigma,\mathbb{X}_\alpha)|^2\right]^{1/2}. 
\end{align*}
Here, we recall that \eqref{eq:10010000} is convergent for $\sigma>1/2$ almost surely. 
Thus we derive 
\begin{gather*}
\mathbf{E}\left[|\zeta(\sigma,\mathbb{X}_\alpha)-\zeta_N(\sigma,\mathbb{X}_\alpha)|^2\right]
= \sum_{m,n>N} \frac{\mathbf{E}[\mathbb{X}_\alpha(m) \overline{\mathbb{X}_\alpha(n)}]}{(m+\alpha)^\sigma (n+\alpha)^\sigma}
= \sum_{n>N} \frac{1}{(n+\alpha)^{2\sigma}}
\end{gather*}
by applying \eqref{eq:10010220}. 
Therefore we obtain
\begin{gather*}
|g(w; \sigma, \mathbb{X}_\alpha)-g_N(w; \sigma, \mathbb{X}_\alpha)|
\ll R\, N^{1/2-\sigma}
\end{gather*}
with the implied constant depending only on $\sigma$, which completes the proof. 
\end{proof}

\begin{proof}[Proof of Theorem \ref{thm:1.2}]
Let $N_0=N_0(\sigma,R,\epsilon)$ be the integer as in Proposition \ref{prop:3.2}. 
Then we choose a positive real number $T_0$ as $T_0=T_0(\sigma,\alpha,N_0,R,\epsilon)$ of Proposition \ref{prop:3.1}. 
With these setting, we obtain the inequality
\begin{align*}
|g_T(w; \sigma, \alpha)
- g(w; \sigma, \mathbb{X}_\alpha)| 
&\leq |g_T(w; \sigma, \alpha)
- g_{T,N_0}(w; \sigma, \alpha)| \\
&\qquad
+ |g_{T,N_0}(w; \sigma, \alpha)-g_{N_0}(w; \sigma, \mathbb{X}_\alpha)| \\
&\qquad
+ |g_{N_0}(w; \sigma, \mathbb{X}_\alpha)-g(w; \sigma, \mathbb{X}_\alpha)| \\
&< 3\epsilon
\end{align*}
for all $T \geq T_0$ in the region $|w| \leq R$. 
Denote by $P_{\sigma,\alpha,T}$ the probability measure defined as in \eqref{eq:10010305}, and denote by $P_{\sigma,\alpha}$ the law of the random variable $\zeta(\sigma,\mathbb{X}_\alpha)$. 
The characteristic functions of these probability measures are represented as 
\begin{gather*}
\int_{\mathbb{C}} \psi_w(z) \,P_{\sigma,\alpha,T}(dz)
= \frac{1}{T} \int_{0}^{T} \psi_w(\zeta(\sigma+it, \alpha)) \,dt
= g_{T}(w; \sigma, \alpha), \\
\int_{\mathbb{C}} \psi_w(z) \,P_{\sigma,\alpha}(dz)
= \mathbf{E}\left[\psi_w(\zeta(\sigma, \mathbb{X}_\alpha))\right]
= g(w; \sigma, \mathbb{X}_\alpha). 
\end{gather*}
Thus the characteristic function of $P_{\sigma,\alpha,T}$ converges to that of $P_{\sigma,\alpha}$ for any $w \in \mathbb{C}$. 
Hence we conclude that the probability measure $P_{\sigma,\alpha,T}$ converges weakly to $P_{\sigma,\alpha}$ as $T \to\infty$ by L\'{e}vy's criterion. 
See \cite[Theorem B.5.1]{Kowalski2021}. 
\end{proof}

\section{A variant of the Cassels lemma}\label{sec:4}
Let $\alpha \in \mathcal{A}$, and denote by $\mathfrak{a}$ the ideal denominator of $\alpha$ in the algebraic field $K=\mathbb{Q}(\alpha)$. 
Then $(n+\alpha) \mathfrak{a}$ is an integral ideal of $K$ for any rational integer $n$. 
The following lemma by Cassels \cite{Cassels1961} is fundamental to study the Hurwitz zeta-function with algebraic irrational parameter. 

\begin{clemma}
Let $\alpha \in \mathcal{A}$. 
Then there exists an integer $N_0=N_0(\alpha)>10^6$ depending on $\alpha$ with the following property. 
Suppose that $N \geq N_0$ and put $M = \lfloor 10^{-6} N \rfloor$. 
Then at least $51M/100$ integers in $N<n \leq N+M$ are such that $(n+\alpha) \mathfrak{a}$ is divisible by a prime ideal $\mathfrak{p}$ which does not divide $(m+\alpha) \mathfrak{a}$ for any integer $0 \leq m \leq N+M$ with $m \neq n$. 
\end{clemma}

As in Lemma \ref{lem:2.4}, we put $\mathcal{L}(N)=\{n \in \mathbb{Z} \mid N<n \leq N \log{N} \}$, and define $\mathcal{K}_\alpha(N)$ as the subset of $\mathcal{L}(N)$ consisting of all integers $n$ such that $(n+\alpha) \mathfrak{a}$ is divisible by a prime ideal $\mathfrak{p}$ not dividing $(m+\alpha) \mathfrak{a}$ for any integer $0 \leq m \leq N \log{N}$ with $m \neq n$. 
The ultimate goal of this section is to prove the following result which is a weighted version of the Cassels lemma. 

\begin{proposition}\label{prop:4.1}
Let $\mathfrak{f}: \mathcal{A} \to \mathbb{R}_{>0}$ be the function defined later as in \eqref{eq:02232117}. 
Then the set $\mathcal{A}_d= \{\alpha \in \mathcal{A} \mid \mathfrak{f}(\alpha) \leq d \}$ satisfies the following property. 
Let $\sigma$ be a fixed real number with $1/2<\sigma<1$. 
For any $d \geq5$, there exists an integer $N_1=N_1(d,\sigma)$ depending only on $d$ and $\sigma$ such that the inequality 
\begin{gather*}
\sum_{n \in \mathcal{K}_\alpha(N)}\frac{1}{(n+\alpha)^\sigma}
> \frac{51}{100} \sum_{n \in \mathcal{L}(N)} \frac{1}{(n+\alpha)^\sigma}
\end{gather*}
holds for all $\alpha \in \mathcal{A}_{d}$ and $N \geq N_1$. 
\end{proposition}

Remark that the integer $N_1$ of Proposition \ref{prop:4.1} is uniform for $\alpha$ in the set $\mathcal{A}_{d}$. 
This fact is used essentially to prove Theorem \ref{thm:1.3} in Section \ref{sec:6}. 
The following proof of Proposition \ref{prop:4.1} is largely based on the method of Cassels \cite{Cassels1961}, Worley \cite{Worley1970}, Mishou \cite{Mishou2008}, and Lee--Mishou \cite{LeeMishou2020}, but we make some modifications. 
See Remark \ref{rem:4.7} for the details of the modifications. 

For any rational integer $n \geq0$, we have
\begin{gather*}
(n+\alpha) \mathfrak{a}
= \prod_{\mathfrak{p} \in J_\alpha} \mathfrak{p}^{u_n(\mathfrak{p})}, 
\end{gather*}
where $J_\alpha$ denotes the set of all prime ideals of $K$, and $u_n(\mathfrak{p})$ are non-negative integers. 
Define $P_\alpha$ as the subset of $J_\alpha$ consisting of prime ideals $\mathfrak{p}$ such that $\mathfrak{p}$ is of the first degree and unambiguous, i.e.\ $\mathrm{N}(\mathfrak{p})=p$ and $\mathfrak{p}^2 \nmid (p)$ for a rational prime $p$. 
Then we rewrite the factorization of $(n+\alpha) \mathfrak{a}$ as 
\begin{align}\label{eq:10071704}
(n+\alpha) \mathfrak{a}
= \mathfrak{b}_n \prod_{\mathfrak{p} \in P_\alpha} \mathfrak{p}^{u_n(\mathfrak{p})}, 
\end{align}
where $\mathfrak{b}_n$ is the integral containing all prime factors of $(n+\alpha) \mathfrak{a}$ which are not in $P_\alpha$. 
In the following, the norm of a prime ideal $\mathfrak{p} \in P_\alpha$ is always denoted by $p$.

\subsection{Preliminary lemmas}\label{sec:4.1}
We begin with showing preliminary lemmas toward the proof of Proposition \ref{prop:4.1}. 
Let $L$ denote the Galois closure of $K$ over $\mathbb{Q}$. 
Choose $\beta \in \mathcal{O}_K$ and $c \in \mathbb{Z}_{\geq1}$ so that $\alpha=\beta/c$. 
Note that we can determine them uniquely from $\alpha$ by fixing $c$ as the minimum one with this property. 
Take an algebraic field $M$ such that $\mathbb{Q} \subset M \subsetneq L$ and $\alpha \notin M$. 
Put $S=\mathcal{O}_M \setminus \mathfrak{q}$ for a prime ideal $\mathfrak{q}$ of $M$. 
Then $S^{-1} \mathcal{O}_L$ is a free $S^{-1}\mathcal{O}_M$-module of rank $\rho:=[L:M]$ since $S^{-1}\mathcal{O}_M$ is a principal ideal ring; see \cite[Ch.~I, Theorem 1]{Lang1994}. 
Taking a basis $\{x_1,\ldots,x_\rho\}$ of $S^{-1} \mathcal{O}_L$ over $S^{-1}\mathcal{O}_M$, we have 
\begin{align}\label{eq:03221652}
1
= a_1 x_1+\cdots+a_\rho x_\rho
\quad\text{and}\quad
\beta
= b_1 x_1+\cdots+b_\rho x_\rho
\end{align}
with some constants $a_\ell, b_\ell \in S^{-1}\mathcal{O}_{M}$. 
Then we define $\mathfrak{g}(\mathfrak{q})$ as the integral ideal of $S^{-1}\mathcal{O}_M$ generated by all elements $a_k b_\ell-a_\ell b_k$ with $k \neq \ell$. 

\begin{lemma}\label{lem:g}
The ideal $\mathfrak{g}(\mathfrak{q}) \subset S^{-1}\mathcal{O}_M$ is independent to the choice of the basis $\{x_1,\ldots,x_\rho\}$. 
Furthermore, we have $\mathfrak{g}(\mathfrak{q})=S^{-1}\mathcal{O}_M$ for all but finitely many prime ideals $\mathfrak{q}$. 
\end{lemma}

\begin{proof}
Take another basis $\{y_1,\ldots,y_\rho\}$ of $S^{-1} \mathcal{O}_L$ over $S^{-1}\mathcal{O}_M$. 
Then there exists a matrix $P=(p_{ij}) \in \mathrm{GL}_\rho(S^{-1}\mathcal{O}_M)$ such that $x_j=p_{1j}y_1+\cdots+p_{\rho j}y_\rho$ for any $j=1,\ldots,\rho$. 
Therefore, if we have 
\begin{align*}
1
= a'_1 y_1+\cdots+a'_\rho y_\rho
\quad\text{and}\quad
\beta
= b'_1 y_1+\cdots+b'_\rho y_\rho, 
\end{align*}
with $a'_j, b'_j \in S^{-1}\mathcal{O}_{M}$, then it yields that 
\begin{align*}
a'_j
= \sum_{\ell=1}^{\rho} p_{\ell j} a_\ell
\quad\text{and}\quad
b'_j
= \sum_{\ell=1}^{\rho} p_{\ell j} b_\ell. 
\end{align*}
Hence $a'_i b'_j-a'_i b'_j$ belongs to the ideal of $S^{-1}\mathcal{O}_M$ generated by all elements $a_k b_\ell-a_\ell b_k$ with $k \neq \ell$. 
It can be similarly shown that $a_k b_\ell-a_\ell b_k$ belongs to the ideal generated by all elements $a'_i b'_j-a'_i b'_j$ with $i \neq j$, and we complete the proof of the first assertion of  the lemma. 

Then we prove the second assertion. 
By the assumption $\alpha \notin M$, we can take an $M$-basis $\{z_1,\ldots,z_\rho\}$ of $L$ such that $z_1=1$ and $z_2=\beta$. 
Furthermore, we take an integer $r \in \mathbb{Z}_{\geq1}$ so that $\{rz_1,\ldots,rz_\rho\} \subset \mathcal{O}_L$. 
Denote by $\Sigma_1$ the set of all prime ideals of $M$ containing $\Delta_{L/M}(rz_1,\ldots,rz_\rho)$, where we define
\begin{align*}
\Delta_{L/M}(w_1,\ldots,w_\rho)
= \det (\mathrm{Tr}_{L/M} (w_i w_j))
\end{align*}
for $w_1,\ldots,w_\rho \in L$. 
Then, for any $\mathfrak{q} \notin \Sigma_1$, we see that $\{rz_1,\ldots,rz_\rho\}$ is a basis of $S^{-1} \mathcal{O}_L$ over $S^{-1}\mathcal{O}_M$ with $S=\mathcal{O}_M \setminus \mathfrak{q}$. 
Indeed, by taking a basis $\{x_1,\ldots,x_\rho\}$ of $S^{-1} \mathcal{O}_L$ over $S^{-1}\mathcal{O}_M$, we have $rz_i=p_{i1}x_1+\cdots+p_{i \rho}x_\rho$ with some matrix $P=(p_{ij})$. 
Then the equality
\begin{align*}
\Delta_{L/M}(rz_1,\ldots,rz_\rho)
= (\det P)^2 \Delta_{L/M}(x_1,\ldots,x_\rho)
\end{align*}
follows. 
Since the prime ideal $\mathfrak{q}$ does not contain the left-hand side, we find that $\det P$ is a unit of $S^{-1}\mathcal{O}_M$, that is, $P$ belongs to $\mathrm{GL}_\rho(S^{-1}\mathcal{O}_M)$. 
Hence $\{rz_1,\ldots,rz_\rho\}$ is a basis of $S^{-1} \mathcal{O}_L$ over $S^{-1}\mathcal{O}_M$. 
Denote by $\Sigma_2$ the set of all prime ideals of $M$ containing $r$. 
Then $1/r$ is a unit of $S^{-1} \mathcal{O}_M$ for any $\mathfrak{q} \notin \Sigma_2$. 
Note that $\Sigma=\Sigma_1 \cup \Sigma_2$ is a finite set. 
From the above, we see that $\mathfrak{g}(\mathfrak{q})=S^{-1}\mathcal{O}_M$ for any $\mathfrak{q} \notin \Sigma$. 
Indeed, using the basis $\{rz_1,\ldots,rz_\rho\}$, we have the representations of $1$ and $\beta$ as in \eqref{eq:03221652}, where
\begin{align*}
a_\ell
= 
\begin{cases}
1/r & \text{if $\ell=1$}, 
\\
0 & \text{otherwise}, 
\end{cases}
\quad\text{and}\quad
b_\ell
= 
\begin{cases}
1/r & \text{if $\ell=2$}, 
\\
0 & \text{otherwise}, 
\end{cases}
\end{align*}
due to $rz_1=r$ and $rz_2=r \beta$. 
Hence we conclude that $\mathfrak{g}(\mathfrak{q})$ is a principal ideal generated by $1/r^2$, which is equal to $S^{-1}\mathcal{O}_M$ as desired. 
\end{proof}

Denote by $G_\mathfrak{q}(\alpha; M)$ the norm of $\mathfrak{g}(\mathfrak{q})$. 
By Lemma \ref{lem:g}, we see that
\begin{align*}
G(\alpha;M)
= \prod_{\mathfrak{q}} G_\mathfrak{q}(\alpha; M) 
\end{align*}
is well-defined, where $\mathfrak{q}$ runs through all prime ideals of $M$.  
Then we define the function $G:\mathcal{A} \to \mathbb{R}_{>0}$ by $G(\alpha)=\max_{M} G(\alpha;M)$, where the maximum is taken over all algebraic fields such that $\mathbb{Q} \subset M \subsetneq L$ and $\alpha \notin M$. 
We also define the function $H:\mathcal{A} \to \mathbb{R}_{>0}$ as
\begin{align*}
H(\alpha)
= \prod_{\sigma \notin \mathrm{Gal}(L/K)} \left|\mathrm{N}_{L/\mathbb{Q}} \left(\beta-\beta^\sigma\right)\right|, 
\end{align*}
where $\sigma$ runs through all elements of $\mathrm{Gal}(L/\mathbb{Q})$ which are not in $\mathrm{Gal}(L/K)$. 
Denote by $\Delta_{L}$ the absolute value of the discriminant of $L$ over $\mathbb{Q}$. 
Let $\overline{|\alpha|}$ be the house of an algebraic number $\alpha$, that is, the maximum value of the absolute values of all conjugates of $\alpha$. 
From the above, we define 
\begin{align}\label{eq:02232117}
\mathfrak{f}(\alpha)
= \max \{G(\alpha), H(\alpha), \Delta_{L}, [L:\mathbb{Q}], \overline{|\alpha|}, \alpha^{-1} \}
\end{align}
for $\alpha \in \mathcal{A}$, and put $\mathcal{A}_d= \{\alpha \in \mathcal{A} \mid \mathfrak{f}(\alpha) \leq d \}$. 

\begin{example}\label{exa:a}
For any integer $n \geq1$, we put
\begin{align*}
\alpha_n
= \frac{n+\phi}{2n+1}
\quad\text{with}\quad
\phi
= \frac{1+\sqrt{5}}{2}. 
\end{align*}
Then we check that $\mathcal{A}_d$ contains $\alpha_n$ for all $n \geq1$ if $d \geq5$ as follows. 
Note that $K=\mathbb{Q}(\alpha_n)$ is a Galois extension over $\mathbb{Q}$, and $\mathcal{O}_K=\mathbb{Z}[\phi]$ in this case. 
We have $\alpha_n=\beta_n/c_n$ with $\beta_n=n+\phi \in \mathcal{O}_K$ and $c_n=2n \in \mathbb{Z}_{\geq1}$. 
Furthermore, if $M$ is an algebraic field such that $\mathbb{Q} \subset M \subsetneq K$, then $M=\mathbb{Q}$ must be satisfied. 
For any rational prime $q$, we see that $\{1,\phi\}$ is a basis of $S^{-1}\mathcal{O}_K$ over $S^{-1}\mathbb{Z}$, where $S=\mathbb{Z} \setminus q \mathbb{Z}$. 
Then we have 
\begin{align*}
1 = 1\cdot1+0\cdot \phi
\quad\text{and}\quad
\beta_n=n \cdot 1+1\cdot \phi, 
\end{align*}
which shows $\mathfrak{g}(q \mathbb{Z})=S^{-1}\mathbb{Z}$ for any rational prime $q$. 
Hence $G(\alpha_n)=1$ for all $n \geq1$ by the definition of $G$. 
We also obtain
\begin{align*}
H(\alpha_n)
= \mathrm{N}_{K/\mathbb{Q}} (\beta_n-\beta_n^\sigma)
= \mathrm{N}_{K/\mathbb{Q}}(\sqrt{5})
= 5, 
\end{align*}
where $\sigma \in \mathrm{Gal}(K/\mathbb{Q})$ such that $\sigma(\sqrt{5})=-\sqrt{5}$. 
Obviously, we have $\Delta_K=5$ and $[K:\mathbb{Q}]=2$. 
Finally, we see that 
\begin{align*}
\overline{|\alpha_n|}
= \max\{|\alpha_n|, |\alpha_n^\sigma|\}
\leq 2
\quad\text{and}\quad
\alpha_n^{-1}
= \frac{2n}{n+\alpha}
\leq 2. 
\end{align*}
From the above, $\mathfrak{f}(\alpha_n)=5$ for all $n \geq1$, and thus $\alpha_n \in \mathcal{A}_d$ if $d \geq5$. 
\end{example}

\begin{lemma}\label{lem:4.2}
Let $d \geq5$. 
Then, with the notation as in \eqref{eq:10071704}, there exists a positive constant $C(d)$ depending only on $d$ such that $\mathrm{N}(\mathfrak{b}_n) \leq C(d)$ holds for any $\alpha \in \mathcal{A}_{d}$ and $n \in \mathbb{Z}_{\geq0}$. 
\end{lemma}

\begin{proof}
Take a prime ideal $\mathfrak{p} \in J_\alpha$ and $v \geq1$ such that $\mathfrak{p}^v \mid \mathfrak{b}_n$. 
We consider an upper bound for the norm of $\mathfrak{p}^v$. 
Let $\mathfrak{p} \cap \mathbb{Z} =p \mathbb{Z}$ for a rational prime $p$. 
Write the factorization of $p$ in $L$ as
\begin{align}\label{eq:03281600}
p \mathcal{O}_L
= \left(\mathfrak{P}_1 \cdots \mathfrak{P}_g\right)^e, 
\end{align}
where $\mathfrak{P}_1, \ldots, \mathfrak{P}_g$ are prime ideals of $L$ with $\mathrm{N}(\mathfrak{P}_k)=p^f$. 
Furthermore, we factorize the prime ideal $\mathfrak{p}$ in $L$ as 
\begin{align}\label{eq:07122257}
\mathfrak{p} \mathcal{O}_L
= \left(\mathfrak{P}_{k_1} \cdots \mathfrak{P}_{k_m}\right)^\eta, 
\end{align}
where $k_1, \ldots, k_m \in \{1, \ldots, g \}$ and $1 \leq \eta \leq e$. 
Denote by $M_j$ the decomposition field of $\mathfrak{P}_{k_j}$ over $\mathbb{Q}$, and put $\mathfrak{q}_j=\mathfrak{P}_{k_j} \cap \mathcal{O}_{M_j}$. 
Then $\mathfrak{P}_{k_j}$ is the unique prime ideal of $L$ lying over $\mathfrak{q}_j$. 
More precisely, we see that
\begin{align}\label{eq:07122258}
\mathrm{N}(\mathfrak{q}_j)
= p
\quad\text{and}\quad
\mathfrak{q}_j \mathcal{O}_L
= \mathfrak{P}_{k_j}^e 
\end{align}
by the properties of decomposition fields. 
See \cite[Ch.~I, Proposition 13]{Lang1994}. 
Note that $\alpha \notin M_j$ must be satisfied. 
Indeed, if $\alpha \in M_j$, then $\mathfrak{q}_j$ is a prime ideal lying over $\mathfrak{p}$, and hence $\mathrm{N}(\mathfrak{p})=p$ by \eqref{eq:07122258}. 
In this case, $\mathfrak{p}^2 \mid p \mathcal{O}_K$ must be satisfied due to $\mathfrak{p} \notin P_\alpha$. 
However, this yields that $\mathfrak{P}_{k_j}^{2e} \mid p \mathcal{O}_L$ by \eqref{eq:07122258}, which contradicts that $p$ is factorized in $L$ as \eqref{eq:03281600}. 
Since $\alpha \notin M_j$, we also obtain that $\mathbb{Q} \subset M_j \subsetneq L$. 
Then, we put $v \eta= eq+r$ with $q \geq0$ and $0 \leq r <e$. 
From \eqref{eq:07122257} and \eqref{eq:07122258}, we obtain
\begin{align}\label{eq:07122339}
\mathfrak{p}^v \mathcal{O}_L
= \left(\mathfrak{P}_{k_1} \cdots \mathfrak{P}_{k_m}\right)^{v\eta} 
= \mathfrak{q}_1^q \mathcal{O}_L \cdots \mathfrak{q}_m^q \mathcal{O}_L 
\cdot \mathfrak{P}_{k_1}^r \cdots \mathfrak{P}_{k_m}^r. 
\end{align}
Hence, taking the norms, we have
\begin{align}\label{eq:07130028}
\mathrm{N}(\mathfrak{p}^v)^{[L:K]}
= \mathrm{N}(\mathfrak{q}_1^q)^{\rho} \cdots \mathrm{N}(\mathfrak{q}_m^q)^{\rho}
\cdot \mathrm{N}(\mathfrak{P}_{k_1})^r \cdots \mathrm{N}(\mathfrak{P}_{k_m})^r, 
\end{align}
where $\rho=[L:M_j]=ef$. 
Then we evaluate $\mathrm{N}(\mathfrak{q}_1^q)^{\rho} \cdots \mathrm{N}(\mathfrak{q}_m^q)^{\rho}$ in \eqref{eq:07130028} by using the function $G(\alpha)$ defined above. 
Note that $\mathfrak{q}_1^q \mathcal{O}_L \cdots \mathfrak{q}_m^q \mathcal{O}_L$ divides $(n+\alpha)\mathfrak{a} \mathcal{O}_L$ by \eqref{eq:07122339} and $\mathfrak{p}^v \mid \mathfrak{b}_n$. 
Then we obtain
\begin{align*}
(nc+\beta) \mathcal{O}_L 
\subset (n+\alpha)\mathfrak{a} \mathcal{O}_L
\subset \mathfrak{q}_j^q \mathcal{O}_L 
\end{align*}
for every $j \in \{1,\ldots,m\}$, where we choose $\beta \in \mathcal{O}_K$ and $c \in \mathbb{Z}_{\geq1}$ so that $\alpha=\beta/c$ and $c$ is the minimum one with this property. 
Put $S=\mathcal{O}_{M_j} \setminus \mathfrak{q}_j$. 
Then $S^{-1}((nc+\beta) \mathcal{O}_L)$ and $S^{-1}(\mathfrak{q}_j^q \mathcal{O}_L)$ are integral ideals of the local ring $S^{-1}\mathcal{O}_L$ which satisfy
\begin{align*}
S^{-1}((nc+\beta) \mathcal{O}_L) 
\subset S^{-1}(\mathfrak{q}_j^q \mathcal{O}_L).
\end{align*}
Therefore, $nc+\beta \in S^{-1}\mathcal{O}_L$ is represented as 
\begin{align*}
nc+\beta
= \sum_{\lambda \in \Lambda} p_\lambda y_\lambda
\end{align*}
for some finite set $\Lambda$, where $p_\lambda \in S^{-1}\mathfrak{q}_j^q$ and $y_\lambda \in S^{-1}\mathcal{O}_L$. 
Furthermore, using a basis $\{x_1,\ldots,x_\rho\}$ of $S^{-1}\mathcal{O}_{L}$ over $S^{-1}\mathcal{O}_{M_j}$, we have 
\begin{align*}
y_\lambda
= \gamma_{1,\lambda} x_1 +\cdots+ \gamma_{\rho,\lambda} x_\rho
\end{align*}
for any $\lambda \in \Lambda$, where $\gamma_{\ell,\lambda} \in S^{-1}\mathcal{O}_{M_j}$ for $\ell \in \{1,\ldots,\rho\}$. 
From these, we obtain the formula
\begin{align*}
nc+\beta
= \left(\sum_{\lambda \in \Lambda} p_\lambda \gamma_{1,\lambda}\right) x_1
+ \cdots 
+ \left(\sum_{\lambda \in \Lambda} p_\lambda \gamma_{\rho,\lambda}\right) x_\rho.  
\end{align*}
On the other hand, we also have
\begin{align*}
nc+\beta
= (nca_1+b_1)x_1+\cdots+(nca_\rho+b_\rho)x_\rho
\end{align*}
by \eqref{eq:03221652}, where $a_\ell, b_\ell \in S^{-1}\mathcal{O}_{M_j}$. 
Comparing them, we find that $nca_\ell+b_\ell$ belongs to $S^{-1} \mathfrak{q}_j^q$ for any $\ell \in \{1,\ldots,\rho\}$. 
This implies that 
\begin{align*}
a_k b_\ell-a_\ell b_k
= a_k(nca_\ell+b_\ell) - a_\ell(nca_k+b_k)
\in S^{-1}\mathfrak{q}_j^q
\end{align*}
for any $k \neq \ell$. 
Hence we see that $S^{-1}\mathfrak{q}_j^q \mid \mathfrak{g}(\mathfrak{q}_j)$, where $\mathfrak{g}(\mathfrak{q}_j)$ is the integral ideal of $S^{-1}\mathcal{O}_{M_j}$ generated by all elements $a_k b_\ell-a_\ell b_k$ with $k \neq \ell$. 
Then we arrive at the inequality
\begin{align*}
\mathrm{N}(\mathfrak{q}_j^q)
= \mathrm{N}(S^{-1}\mathfrak{q}_j^q)
\leq G(\alpha)
\end{align*}
by the definition of $G(\alpha)$. 
As a result, we obtain 
\begin{align}\label{eq:07130030}
\mathrm{N}(\mathfrak{q}_1^q)^{\rho} \cdots \mathrm{N}(\mathfrak{q}_m^q)^{\rho}
\leq G(\alpha)^{\rho m}
\leq G(\alpha)^{[L:\mathbb{Q}]}
\end{align}
due to $\rho=ef$ and $[L:\mathbb{Q}]=efg$. 
Furthermore, $\mathrm{N}(\mathfrak{P}_{k_1})^r \cdots \mathrm{N}(\mathfrak{P}_{k_m})^r$ in \eqref{eq:07130028} is evaluated as follows. 
If $e \geq2$, then we recall that $p \mid \Delta_{L}$ is satisfied. 
In this case, $\mathrm{N}(\mathfrak{P}_{k_j})^r = p^{fr} \leq \Delta_{L}^{\rho}$ follows. 
If $e=1$, then we have $\mathrm{N}(\mathfrak{P}_{k_j})^r=1$ since $r$ must be equal to $0$. 
As a result, we obtain
\begin{align}\label{eq:07130029}
\mathrm{N}(\mathfrak{P}_{k_1})^r \cdots \mathrm{N}(\mathfrak{P}_{k_m})^r
\leq \Delta_{L}^{\rho m}
\leq \Delta_{L}^{[L:\mathbb{Q}]} 
\end{align}
due to $[L:\mathbb{Q}]=\rho g$. 
By \eqref{eq:07130028}, \eqref{eq:07130030}, and \eqref{eq:07130029}, we arrive at
\begin{align*}
\mathrm{N}(\mathfrak{p}^v)
\leq \left(G(\alpha) \Delta_{L}\right)^{[L:\mathbb{Q}]}
\leq d^{2d} 
=: C_1(d) 
\end{align*}
for any $\alpha \in \mathcal{A}_d$ by the definition of $\mathcal{A}_d$. 
Then the desired result immediately follows. 
Indeed, we have  
\begin{align*}
\mathrm{N}(\mathfrak{b}_n)
= \prod_{\mathfrak{p} \mid \mathfrak{b}_n} \mathrm{N}(\mathfrak{p}^{\nu_\mathfrak{p}(\mathfrak{b}_n)})
\leq C_1(d)^{\pi_K(C_1(d))}, 
\end{align*}
where $\pi_K(x)$ denotes the number of prime ideals $\mathfrak{p}$ of $K$ such that $\mathrm{N}(\mathfrak{p}) \leq x$. 
By the result of Lagarias--Montgomery--Odlyzko \cite[Theorem 1.4]{LagariasMontgomeryOdlyzko1979}, we have
\begin{align}\label{eq:03290039}
\pi_K(x)
\leq \frac{Ax}{\log{x}}
\quad\text{for}\quad
x 
> \exp\left(B (\log{\Delta_L}) (\log\log{\Delta_L}) (\log\log\log{\Delta_L e^{20}})\right), 
\end{align}
where $A$ and $B$ are positive absolute constants. 
Hence $C_1(d)^{\pi_K(C_1(d))}$ can be bounded by a positive constant $C(d)$ which depends only on $d$. 
\end{proof}

\begin{lemma}\label{lem:4.2'}
Let $\alpha \in \mathcal{A}_d$ with $d \geq5$. 
Suppose that $\mathfrak{p}_1, \mathfrak{p}_2 \in P_\alpha$ are distinct prime factors of $(n+\alpha)\mathfrak{a}$ for some $n \in \mathbb{Z}$ and have the same norm. 
Then we have $p:=\mathrm{N}(\mathfrak{p}_1)=\mathrm{N}(\mathfrak{p}_2) \leq d$. 
\end{lemma}

\begin{proof}
If $p$ is factorized in $L$ as in \eqref{eq:03281600} with $e \geq2$, then $p \leq \Delta_L$ follows. 
Then we below assume $e=1$. 
We factorize $\mathfrak{p}_1$ and $\mathfrak{p}_2$ in $L$ as 
\begin{align*}
\mathfrak{p}_1 \mathcal{O}_L
= \mathfrak{P}_{k_1} \cdots \mathfrak{P}_{k_m}
\quad\text{and}\quad
\mathfrak{p}_2 \mathcal{O}_L
= \mathfrak{P}_{\ell_1} \cdots \mathfrak{P}_{\ell_\mu} 
\end{align*}
similarly to \eqref{eq:07122257}. 
Here, we may assume $k_m \geq \ell_\mu$ without loss of generality. 
Since $\mathfrak{p}_1 \mathcal{O}_L$ and $\mathfrak{p}_2 \mathcal{O}_L$ are distinct ideals, there exists $k \in \{k_1,\ldots,k_m\}$ such that $\mathfrak{P}_k \nmid \mathfrak{p}_2 \mathcal{O}_L$. 
Take $\ell \in \{\ell_1,\ldots,\ell_\mu\}$ arbitrarily. 
If we put $\mathfrak{P}_k=\mathfrak{P}_{\ell}^\sigma$ for some $\sigma \in \mathrm{Gal}(L/\mathbb{Q})$, then we must have $\sigma \notin \mathrm{Gal}(L/K)$ due to $\mathfrak{P}_k \nmid \mathfrak{p}_2 \mathcal{O}_L$ and $\mathfrak{P}_\ell \mid \mathfrak{p}_2 \mathcal{O}_L$. 
Recall that both $\mathfrak{P}_k$ and $\mathfrak{P}_\ell$ divide $(nc+\beta) \mathcal{O}_L$, where $\beta \in \mathcal{O}_K$ and $c \in \mathbb{Z}_{\geq1}$ are as in the proof of Lemma \ref{lem:4.2}. 
Thus we obtain 
\begin{align*}
\mathfrak{P}_k \mid (nc+\beta) \mathcal{O}_L
\quad\text{and}\quad
\mathfrak{P}_k \mid (nc+\beta^\sigma) \mathcal{O}_L. 
\end{align*}
These yield that $\mathfrak{P}_k \mid (\beta-\beta^\sigma) \mathcal{O}_L$, and furthermore, 
\begin{align*}
\mathrm{N}(\mathfrak{P}_k)
\leq \left|\mathrm{N}_{L/\mathbb{Q}} (\beta-\beta^\sigma)\right|
\leq H(\alpha)
\end{align*}
by the definition of $H(\alpha)$. 
Since $\mathfrak{P}_k$ is a prime ideal of $L$ lying above $p$, we have $\mathrm{N}(\mathfrak{P}_k)=p^f \geq p$ for some $f \geq 1$. 
As a result, we have the upper bound
\begin{align*}
p 
\leq \max\{\Delta_L, H(\alpha) \}
\leq d
\end{align*}
for any $\alpha \in \mathcal{A}_d$ by the definition of $\mathcal{A}_d$. 
\end{proof}

\begin{lemma}\label{lem:4.3}
Let $\mathfrak{p} \in P_\alpha$ and $v \geq1$. 
If $\mathfrak{p}^v \mid (m+\alpha)\mathfrak{a}$ and $\mathfrak{p}^v \mid (n+\alpha)\mathfrak{a}$ for $m,n \in \mathbb{Z}$, then we have $m \equiv n ~(\bmod\,{p^v})$. 
\end{lemma}

\begin{proof}
Put $m-n=p^w A$ with some $w \in \mathbb{Z}_{\geq0}$ and $A \in \mathbb{Z}_{\geq1}$ such that $(p,A)=1$. 
By the assumption, we have $(m+\alpha)\mathfrak{a} \subset \mathfrak{p}^v$ and $(n+\alpha)\mathfrak{a} \subset \mathfrak{p}^v$. 
Thus we see that $(m-n)\mathfrak{a} \subset \mathfrak{p}^v$ holds, which yields
\begin{align*}
\mathfrak{p}^v \mid (m-n) 
\end{align*}
due to $\mathfrak{p} \nmid \mathfrak{a}$. 
Therefore, $\mathfrak{p}^v \mid (p)^w (A)$ follows. 
We deduce that $\mathfrak{p} \nmid (A)$ from $(p,A)=1$ by taking the norms. 
As a result, we obtain 
\begin{align*}
\mathfrak{p}^v \mid (p)^w. 
\end{align*}
Recall that $\mathfrak{p}$ is unambiguous, i.e.\ $\mathfrak{p}^2 \nmid (p)$. 
Hence we see that $w \geq v$, which implies $p^v \mid (m-n)$ as desired. 
\end{proof}

\begin{lemma}\label{lem:4.4}
Let $1/2<\sigma<1$ be a fixed real number. 
For any $0<\alpha \leq1$, we have
\begin{align*}
\sum_{\substack{ n \in \mathcal{L}(N) \\ n \equiv a \hspace{-2truemm} \pmod{q}}} \frac{1}{(n+\alpha)^\sigma}
&= \frac{1}{q} \sum_{n \in \mathcal{L}(N)} \frac{1}{(n+\alpha)^\sigma}
+ O(N^{-\sigma}) 
\end{align*}
with arbitrary integers $q \geq1$ and $0 \leq a<q$, where the implied constant is absolute. 
\end{lemma}

\begin{proof}
Let $f(x)=(x+\alpha)^{-\sigma}$ and $c_n=1$ if $n \equiv a \pmod{q}$ and $c_n=0$ otherwise. 
Then we have 
\begin{gather*}
\mathcal{C}(x)
:= \sum_{N<n \leq x} c_n
= \frac{1}{q}(x-N) + O(1). 
\end{gather*}
By the partial summation formula \cite[Theorem 1 on p.\,326]{KaratsubaVoronin1992}, we obtain
\begin{align*}
&\sum_{\substack{ n \in \mathcal{L}(N) \\ n \equiv a \hspace{-2truemm} \pmod{q}}} \frac{1}{(n+\alpha)^\sigma}
= -\int_{N}^{N \log{N}} \mathcal{C}(x)f'(x) \,dx
+ \mathcal{C}(N \log{N}) f(N \log{N}) \\
&= \frac{1}{q} \left\{ \sigma \int_{N}^{N \log{N}} \frac{x-N}{(x+\alpha)^{\sigma+1}} \,dx 
+ \frac{N \log{N}-N}{(N \log{N}+\alpha)^\sigma} \right\} 
+ O\left(\frac{1}{(N+\alpha)^\sigma}\right),  \nonumber 
\end{align*}
where the implied constant is absolute. 
Letting $q=1$, we also obtain
\begin{align*}
&\sum_{n \in \mathcal{L}(N)} \frac{1}{(n+\alpha)^\sigma} \\
&= \sigma \int_{N}^{N \log{N}} \frac{x-N}{(x+\alpha)^{\sigma+1}} \,dx 
+ \frac{N \log{N}-N}{(N \log{N}+\alpha)^\sigma} 
+ O\left(\frac{1}{(N+\alpha)^\sigma}\right). 
\end{align*}
Comparing these formulas, we obtain the first formula. 
\end{proof}

\subsection{Cassels method}\label{sec:4.2}
Let $1/2<\sigma<1$ be a fixed real number. 
Define 
\begin{gather*}
\mathfrak{S}_\alpha(N)
= \left\{n \in \mathcal{L}(N) ~\middle|~ 
\text{$p^{u_n(\mathfrak{p})} \leq N \log{N}$ for all $\mathfrak{p} \in P_\alpha$} \right\}. 
\end{gather*}
Then we put 
\begin{gather}\label{eq:10080022}
\sum_{n \in \mathfrak{S}_\alpha(N)} \frac{1}{(n+\alpha)^\sigma}
= \rho \sum_{n \in \mathcal{L}(N)} \frac{1}{(n+\alpha)^\sigma}, 
\end{gather}
where $\rho=\rho(\sigma,\alpha,N)$ is a real number such that $0<\rho \leq1$. 
In this subsection, we prove that  $\rho \leq 0.48+o(1)$ holds as $N \to\infty$ uniformly for $\alpha \in \mathcal{A}_{d}$ according to the method of Cassels \cite{Cassels1961}. 
For $n \in \mathcal{L}(N)$, we define 
\begin{gather}\label{eq:10080039}
\sigma(n)
= \mathop{ \sum_{\mathfrak{p} \in P_\alpha} \sum_{v=1}^{\infty} } \limits_{p^v \leq N \log{N}} \phi(\mathfrak{p}^v, n), 
\end{gather}
where $\phi(\mathfrak{p}^v, n)=\log{p}$ if $\mathfrak{p}^v \mid (n+\alpha) \mathfrak{a}$, and $\phi(\mathfrak{p}^v, n)=0$ otherwise. 
Then we have the following result. 

\begin{proposition}\label{prop:4.5}
Let $1/2<\sigma<1$ be a fixed real number. 
Then we have 
\begin{gather}\label{eq:10092331}
\sum_{n \in \mathfrak{S}_\alpha(N)} \frac{\sigma(n)}{(n+\alpha)^\sigma}
\geq (2\rho+o(1)) \log(N \log{N}) \sum_{n \in \mathcal{L}(N)} \frac{1}{(n+\alpha)^\sigma}
\end{gather}
as $N \to\infty$ uniformly for $\alpha \in \mathcal{A}_{d}$, where $\rho=\rho(\sigma,\alpha,N)$ satisfies \eqref{eq:10080022}. 
\end{proposition}

\begin{proof}
For $\alpha \in \mathcal{A}_{d}$ and $n \in \mathbb{Z}$ with $n \geq 2d$, we have 
\begin{gather*}
\mathrm{N}((n+\alpha) \mathfrak{a})
= \mathrm{N}_{K/\mathbb{Q}} (n+\alpha) \, \mathrm{N}(\mathfrak{a})
\geq \left(n- \overline{|\alpha|} \right)^{\deg(\alpha)}
\geq \frac{n^2}{4}
\end{gather*}
since $\overline{|\alpha|} \leq d$ is satisfied for any $\alpha \in \mathcal{A}_{d}$ by definition. 
Applying this lower bound and Lemma \ref{lem:4.2}, we deduce from \eqref{eq:10071704} the inequality
\begin{gather*}
\sum_{\mathfrak{p} \in P_\alpha} u_n(\mathfrak{p}) \log{p}
\geq 2\log{n}+O(1), 
\end{gather*}
where the implied constant depends only on $d$. 
For $n \in \mathfrak{S}_\alpha(N)$ with $N \geq 2d$, it further yields that $\sigma(n) \geq 2\log{n}+O(1)$ by the definition of $\phi(\mathfrak{p}^v, n)$. 
Therefore, 
\begin{align*}
\sum_{n \in \mathfrak{S}_\alpha(N)} \frac{\sigma(n)}{(n+\alpha)^\sigma}
&\geq \sum_{n \in \mathfrak{S}_\alpha(N)} \frac{2 \log{n}+O(1)}{(n+\alpha)^\sigma} \\
&\geq (2\rho+o(1)) \log{N} \sum_{n \in \mathcal{L}(N)} \frac{1}{(n+\alpha)^\sigma}. 
\end{align*}
Using that $\log{N}= \log(N \log{N}) (1+o(1))$ as $N \to\infty$, we obtain the conclusion. 
\end{proof}

Furthermore, we divide $\sigma(n)$ as in \eqref{eq:10080039} into the following three parts: 
\begin{gather*}
\sigma_1(n)
= \mathop{ \sum_{ \mathfrak{p} \in P_\alpha} \sum_{v=1}^{\infty} } \limits_{p^v \leq N \log{N}} 
\phi(\mathfrak{p}^v, n), 
\qquad
\sigma_2(n) 
= \sum_{\substack{ \mathfrak{p} \in P_\alpha \\ \sqrt{N \log{N}} < p \leq N \log{N}}} 
\phi(\mathfrak{p}, n), \\
\sigma_3(n)
= \sum_{\substack{ \mathfrak{p} \in P_\alpha \\ p \leq \sqrt{N \log{N}}}} 
\phi(\mathfrak{p}, n). 
\end{gather*}

\begin{proposition}\label{prop:4.6}
Let $1/2<\sigma<1$ be a fixed real number. 
Then we have 
\begin{align}
\sum_{n \in \mathfrak{S}_\alpha(N)} \frac{\sigma_1(n)}{(n+\alpha)^\sigma}
&= o(\log(N \log{N})) \sum_{n \in \mathcal{L}(N)} \frac{1}{(n+\alpha)^\sigma}, 
\label{eq:10080100} \\
\sum_{n \in \mathfrak{S}_\alpha(N)} \frac{\sigma_2(n)}{(n+\alpha)^\sigma}
&\leq \left(\frac{1}{2}+o(1)\right) \log(N \log{N}) \sum_{n \in \mathcal{L}(N)} \frac{1}{(n+\alpha)^\sigma}, 
\label{eq:10080101} \\
\sum_{n \in \mathfrak{S}_\alpha(N)} \frac{\sigma_3(n)^2}{(n+\alpha)^\sigma}
&\leq \left(\frac{3}{8}+o(1)\right) \log^2(N \log{N})  
\sum_{n \in \mathcal{L}(N)} \frac{1}{(n+\alpha)^\sigma}
\label{eq:10080102}
\end{align}
as $N \to\infty$ uniformly for $\alpha \in \mathcal{A}_{d}$. 
\end{proposition}

\begin{proof}
By the definitions $\sigma_1(n)$ and $\sigma_2(n)$, we have 
\begin{gather*}
\sum_{n \in \mathfrak{S}_\alpha(N)} \frac{\sigma_1(n)}{(n+\alpha)^\sigma}
\leq \sum_{n \in \mathcal{L}(N)} \frac{\sigma_1(n)}{(n+\alpha)^\sigma} 
= \mathop{ \sum_{ \mathfrak{p} \in P_\alpha} \sum_{v=2}^{\infty} } \limits_{p^v \leq N \log{N}} 
\sum_{n \in \mathcal{L}(N)} \frac{\phi(\mathfrak{p}^v, n)}{(n+\alpha)^\sigma},  
\end{gather*}
and similarly, 
\begin{align*}
\sum_{n \in \mathfrak{S}_\alpha(N)} \frac{\sigma_2(n)}{(n+\alpha)^\sigma}
& \leq \sum_{\substack{ \mathfrak{p} \in P_\alpha \\ \sqrt{N \log{N}} < p \leq N \log{N} }} 
\sum_{n \in \mathcal{L}(N)} \frac{\phi(\mathfrak{p}, n)}{(n+\alpha)^\sigma}. 
\end{align*}
Then we evaluate the inner sums in a way similar to \cite[$(29)$]{Cassels1961} by applying Lemmas \ref{lem:4.3} and \ref{lem:4.4}. 
As a result, we obtain
\begin{align*}
\sum_{n \in \mathcal{L}(N)} \frac{\phi(\mathfrak{p}^v, n)}{(n+\alpha)^\sigma}
&\leq (\log{p})\left\{\frac{1}{p^v} \sum_{n \in \mathcal{L}(N)} \frac{1}{(n+\alpha)^\sigma} 
+ O(N^{-\sigma})\right\}. 
\end{align*}
By using this, \eqref{eq:10080100} and \eqref{eq:10080101} follow along the same lines as \cite[$(32), (33)$]{Cassels1961}. 
On the other hand, the proof of \eqref{eq:10080102} requires more careful treatment. 
By the definition of $\sigma_3(n)$, we have 
\begin{align}\label{eq:03290022}
\sum_{n \in \mathfrak{S}_\alpha(N)} \frac{\sigma_3(n)^2}{(n+\alpha)^\sigma}
& \leq \sum_{\substack{\mathfrak{p} \in P_\alpha \\ p \leq \sqrt{N \log{N}} }}
\sum_{n \in \mathcal{L}(N)} \frac{\phi(\mathfrak{p}, n)^2}{(n+\alpha)^\sigma} \\
&+ \sum_{\substack{\mathfrak{p}_1 \neq \mathfrak{p}_2 \in P_\alpha \\ p_1 \neq p_2 \leq \sqrt{N \log{N}} }}
\sum_{n \in \mathcal{L}(N)} \frac{\phi(\mathfrak{p}_1, n)\phi(\mathfrak{p}_2, n)}{(n+\alpha)^\sigma} \nonumber \\
&+ \sum_{\substack{\mathfrak{p}_1 \neq \mathfrak{p}_2 \in P_\alpha \\ p_1=p_2 \leq \sqrt{N \log{N}} }}
\sum_{n \in \mathcal{L}(N)} \frac{\phi(\mathfrak{p}_1, n)\phi(\mathfrak{p}_2, n)}{(n+\alpha)^\sigma}, \nonumber
\end{align}
where $p_1=\mathrm{N}(\mathfrak{p}_1)$ and $p_2=\mathrm{N}(\mathfrak{p}_2)$. 
The first sum of the right-hand side can be estimated by an argument similar to the proof of \eqref{eq:10080100} and \eqref{eq:10080101}. 
Furthermore, along the same line as \cite[$(30)$]{Cassels1961}, we have 
\begin{align*}
&\sum_{n \in \mathcal{L}(N)} \frac{\phi(\mathfrak{p}_1, n)\phi(\mathfrak{p}_2, n)}{(n+\alpha)^\sigma} \\
&\leq (\log{p}_1)(\log{p}_2) 
\left\{\frac{1}{p_1 p_2} \sum_{n \in \mathcal{L}(N)} \frac{1}{(n+\alpha)^\sigma} + O(N^{-\sigma})\right\}
\end{align*}
provided that $p_1 \neq p_2$, using the Chinese remainder theorem. 
Hence, the first two sums on the right-hand side of \eqref{eq:03290022} is less than or equal to 
\begin{align*}
\left(\frac{3}{8}+o(1)\right) \log^2(N \log{N})  
\sum_{n \in \mathcal{L}(N)} \frac{1}{(n+\alpha)^\sigma}
\end{align*}
similarly to \cite[$(34)$]{Cassels1961}. 
Finally, we apply Lemma \ref{lem:4.2'} to see that the third sum on the right-hand side is less than or equal to 
\begin{align*}
\sum_{\substack{\mathfrak{p}_1 \neq \mathfrak{p}_2 \in P_\alpha \\ p_1=p_2 \leq d }}
\sum_{n \in \mathcal{L}(N)} \frac{(\log{p}_1)(\log{p}_2)}{(n+\alpha)^\sigma}
\leq \pi_K(d)^2 (\log{d})^2
\sum_{n \in \mathcal{L}(N)} \frac{1}{(n+\alpha)^\sigma}. 
\end{align*}
From these results together with \eqref{eq:03290039}, we derive that \eqref{eq:10080102} follows. 
\end{proof}

Now we are ready to show $\rho \leq 0.48+o(1)$ as $N \to\infty$ uniformly for $\alpha \in \mathcal{A}_{d}$. 
By \eqref{eq:10092331}, \eqref{eq:10080100}, \eqref{eq:10080101}, we have the lower bound 
\begin{gather*}
\sum_{n \in \mathfrak{S}_\alpha(N)} \frac{\sigma_3(n)}{(n+\alpha)^\sigma}
\geq \left(2\rho - \frac{1}{2} +o(1)\right) \log(N \log{N}) 
\sum_{n \in \mathcal{L}(N)} \frac{1}{(n+\alpha)^\sigma} 
\end{gather*}
which corresponds to \cite[$(35)$]{Cassels1961}. 
On the other hand, we deduce from \eqref{eq:10080102} the upper bound
\begin{align*}
\sum_{n \in \mathfrak{S}_\alpha(N)} \frac{\sigma_3(n)}{(n+\alpha)^\sigma}
\leq \left(\sqrt{\frac{3\rho}{8}} +o(1)\right) \log(N \log{N}) 
\sum_{n \in \mathcal{L}(N)} \frac{1}{(n+\alpha)^\sigma} 
\end{align*}
by using the Cauchy--Schwarz inequality. 
Comparing the lower and upper bounds, we obtain the inequality
\begin{gather*}
2\rho-\frac{1}{2} 
\leq \sqrt{\frac{3 \rho}{8}}
+ o(1), 
\end{gather*}
which yields $\rho \leq 0.48+o(1)$ by a similar argument used to deduce \cite[$(37)$]{Cassels1961}. 
Then, we complete the proof of Proposition \ref{prop:4.1}. 

\begin{proof}[Proof of Proposition \ref{prop:4.1}]
With the notation above, we define 
\begin{align*}
\mathfrak{T}_\alpha(N) 
= \mathcal{L}(N) \setminus \mathfrak{S}_\alpha(N). 
\end{align*}
Then, for any $n \in \mathfrak{T}_\alpha(N)$, there exists a prime ideal $\mathfrak{p} \in P_\alpha$ such that
\begin{gather}\label{eq:10100000}
\mathfrak{p}^v 
\mid (n+\alpha) \mathfrak{a}
\quad\text{and}\quad
p^v
> N \log{N}
\end{gather}
for some integer $v \geq1$. 
By an argument similar to \cite[$(38)$]{Cassels1961} and its sequel, it readily follows that
\begin{gather*}
\sum_{\substack{n \in \mathfrak{T}_\alpha(N) \\ \text{$\mathfrak{p}$ in \eqref{eq:10100000} satisfies $p \leq N \log{N}$}}}
\frac{1}{(n+\alpha)^\sigma}
\leq N^{-\sigma} \sum_{\substack{\mathfrak{p} \in P_\alpha \\ p \leq N \log{N}}} 1
\ll N^{1-\sigma} 
\end{gather*}
and 
\begin{align*}
\left\{n \in \mathfrak{T}_\alpha(N) ~\middle|~ 
\text{$\mathfrak{p}$ in \eqref{eq:10100000} satisfies $p> N \log{N}$} \right\}
\subset \mathcal{K}_\alpha(N). 
\end{align*}
Therefore, we obtain the desired result by $\rho \leq 0.48+o(1)$. 
\end{proof}

\begin{remark}\label{rem:4.7}
Cassels originally considered in \cite{Cassels1961} the factorization of the integral ideal $(n+\alpha) \mathfrak{a}$ similar to \eqref{eq:10071704} but the set $P_\alpha$ is replaced by $P'_\alpha$ containing the prime ideals $\mathfrak{p}$ such that 
\begin{itemize}
\item[$(\mathrm{i})$]
$\mathfrak{p}$ is of the first degree and unambiguous; 
\item[$(\mathrm{ii})$]
for any integer $m$, if $\mathfrak{p} \mid (m+\alpha) \mathfrak{a}$ then $\mathfrak{p}' \nmid (m+\alpha) \mathfrak{a}$ for any prime ideal $\mathfrak{p}' \neq \mathfrak{p}$ with $\mathrm{N}(\mathfrak{p})=\mathrm{N}(\mathfrak{p}')$. 
\end{itemize}
Then it was claimed in \cite[p.\ 181]{Cassels1961} that the norm of $\mathfrak{b}'_n$ is bounded for $n \in \mathbb{Z}$, where the ideal $\mathfrak{b}'_n$ is similar to $\mathfrak{b}_n$ as in \eqref{eq:10071704}, that is, it contains all prime factors of $(n+\alpha) \mathfrak{a}$ which are not in $P'_\alpha$. 
However, we have a counterexample to the boundedness of $\mathrm{N}(\mathfrak{b}'_n)$. 
Let $\alpha=7\sqrt{2}$. 
Then we have $K=\mathbb{Q}(\sqrt{2})$ and $\mathfrak{a}=(1)$. 
Note that $\mathfrak{p}_1=(3+\sqrt{2})$ and $\mathfrak{p}_2=(3-\sqrt{2})$ are prime ideals of $K$ satisfying $\mathrm{N}(\mathfrak{p}_1)=\mathrm{N}(\mathfrak{p}_2)=7$ and $(7)=\mathfrak{p}_1 \mathfrak{p}_2$. 
For any $k \in \mathbb{Z}$, we have
\begin{gather*}
(7k+\alpha)
= \mathfrak{p}_1 \mathfrak{p}_2 (k+\sqrt{2}). 
\end{gather*}
It implies that the above condition $(\mathrm{ii}')$ is not satisfied for $\mathfrak{p}_1$ and $\mathfrak{p}_2$. 
Therefore, they are prime factors of the ideal $\mathfrak{b}'_n$ for $n=7k$. 
The norm of $(k+\sqrt{2})$ is equal to $k^2-2$. 
By a simple application of Hensel's lemma, for any $v \geq1$, there exists an integer $k_v$ such that $k_v^2-2$ is divisible by $7^v$. 
Hence $(k_v+\sqrt{2})$ is divisible by either $\mathfrak{p}_1^v$ or $\mathfrak{p}_2^v$ for any $v \geq0$. 
As a result, either one of 
\begin{gather*}
\mathfrak{p}_1^{v} \mid \mathfrak{b}'_n
\quad\text{and}\quad
\mathfrak{p}_2^{v} \mid \mathfrak{b}'_n
\end{gather*}
is valid for $n=7k_v$, and we have $\mathrm{N}(\mathfrak{b}'_n) \geq 7^v$ in either case. 
Therefore $\mathrm{N}(\mathfrak{b}'_n)$ is unbounded as $n$ varies over rational integers. 

Recall that condition $(\mathrm{ii})$ was used only in \cite[$(34)$]{Cassels1961} to see that there is no contribution of terms for $\mathfrak{p}_1 \neq \mathfrak{p}_2$ and $p_1=p_2$ in the sum containing $\sigma_3(n)^2$ which corresponds to \eqref{eq:03290022}. 
This suggests that we do not need here to take into account the higher than $1$ powers of prime ideals by the definition of $\sigma_3(n)$. 
Note that the above counterexample is based on high powers of finitely many prime ideals. 
Then the original proof by Cassels may be recovered by some modification on the classification of prime ideals. 
For example, one can probably recover the proof by modifying condition $(\mathrm{ii})$ as 
\begin{itemize}
\item[$(\mathrm{ii}')$]
if $u_n(\mathfrak{p})=1$, then $u_n(\mathfrak{p}') \neq 1$ for any $\mathfrak{p}' \neq \mathfrak{p}$ with the same norm. 
\end{itemize}
This formulation appears to match the original conception of the proof by Cassels. 
On the other hand, as seen above, we can use Lemma \ref{lem:4.2'} in place of condition $(\mathrm{ii})$ to show that the contribution of the terms for $\mathfrak{p}_1 \neq \mathfrak{p}_2$ and $p_1=p_2$ in \eqref{eq:03290022} is not zero but can be negligible. 
This seems like a simpler solution, and for this reason, we decided to use only $(\mathrm{i})$ in the classification of prime ideals in this paper. 

Lastly, we note that the work of Worley \cite{Worley1970}, Mishou \cite{Mishou2008}, and Lee--Mishou \cite{LeeMishou2020} is based on the same erroneous characterization of the prime ideals as Cassels \cite{Cassels1961}, but we can recover the proofs by the modification explained in this paper. 
\end{remark}

\section{Proof of Theorem \ref{thm:1.3}}\label{sec:6}

\subsection{Support of random variables}\label{sec:6.1}
Let $\mu$ be a probability measure on $(\mathbb{C}, \mathcal{B}(\mathbb{C}))$. 
The support of $\mu$ is defined as
\begin{gather*}
\supp(\mu)
= \left\{ z \in \mathbb{C} ~\middle|~ \text{$\mu(A)>0$ holds for any set $A$ with $z \in A^i$} \right\}, 
\end{gather*}
where $A^i$ denotes the interior of $A$. 
For a random variable $\mathcal{X}$, we define $\supp(\mathcal{X})$ as the support of the law of $\mathcal{X}$. 
Then $\supp(\mathcal{X})$ is a non-empty closed set of $\mathbb{C}$. 
For closed subsets $A_j \subset \mathbb{C}$ for $1 \leq j \leq n$, we denote by $A_1+\cdots+A_n$ the set of all points $a_1+\cdots+a_n$ with $a_j \in A_j$ for $1 \leq j \leq n$. 
If every $A_j$ is bounded, then we see that $A_1+\cdots+A_n$ is again closed. 
Therefore, if $\mathcal{X}_1, \ldots, \mathcal{X}_n$ are independent random variables such that every $\supp(\mathcal{X}_j)$ is bounded, then 
\begin{gather}\label{eq:10102313}
\supp(S_n)
= \overline{\supp(\mathcal{X}_1)+ \cdots +\supp(\mathcal{X}_n)}
= \supp(\mathcal{X}_1)+ \cdots +\supp(\mathcal{X}_n)
\end{gather}
holds for $S_n= \mathcal{X}_1+\cdots+\mathcal{X}_n$; see the proof of \cite[Proposition B.10.8]{Kowalski2021}. 
The purpose of this subsection is to study the support of $\zeta_N(\sigma,\mathbb{X}_\alpha)$ defined as \eqref{eq:10121134}. 

\begin{proposition}\label{prop:6.1}
Let $\mathcal{A}_{d}$ denote the set of Proposition \ref{prop:4.1}. 
Let $1/2<\sigma<1$ be a fixed real number. 
For any $z_0 \in \mathbb{C}$, there exists an integer $N_2=N_2(d,\sigma,z_0)$ depending only on $d,\sigma$, and $z_0$ such that 
\begin{gather}\label{eq:10110027}
z_0 
\in \supp \left(\zeta_N(\sigma,\mathbb{X}_\alpha)\right)
\end{gather}
holds for all $\alpha \in \mathcal{A}_{d}$ and $N \geq N_2$. 
\end{proposition}

\begin{proof}
Let $\mathcal{K}_\alpha(N) = \left\{ n_1, \ldots, n_k \right\}$ and $\mathcal{M}(N) \setminus \mathcal{K}_\alpha(N)= \left\{ n_{k+1}, \ldots, n_\ell \right\}$ be as in Lemma \ref{lem:2.4}. 
Then the random variables 
\begin{gather}\label{eq:10110025}
\frac{\mathbb{X}_\alpha(n_1)}{(n_1+\alpha)^\sigma}, \ldots, \frac{\mathbb{X}_\alpha(n_k)}{(n_k+\alpha)^\sigma},\,
\sum_{j=k+1}^{\ell} \frac{\mathbb{X}_\alpha(n_j)}{(n_j+\alpha)^\sigma}
\end{gather}
are independent.
Here, we note that the support of the random variable $\mathbb{X}_\alpha(n)/(n+\alpha)^\sigma$ is equal to the circle $\{z : |z|=(n+\alpha)^{-\sigma}\}$. 
Thus we apply \eqref{eq:10102313} to deduce
\begin{align*}
\supp \left( \sum_{n \in \mathcal{M}(N)} \frac{\mathbb{X}_\alpha(n)}{(n+\alpha)^\sigma} \right) 
&= \supp \left(\frac{\mathbb{X}_\alpha(n_1)}{(n_1+\alpha)^\sigma}\right)
+ \cdots +\supp \left(\frac{\mathbb{X}_\alpha(n_k)}{(n_k+\alpha)^\sigma}\right) \\
&\qquad
+ \supp \left(\sum_{j=k+1}^{\ell} \frac{\mathbb{X}_\alpha(n_j)}{(n_j+\alpha)^\sigma}\right). 
\end{align*}
Furthermore, we deduce from \cite[Theorem 9]{JessenWintner1935} that
\begin{gather*}
\supp \left(\frac{\mathbb{X}_\alpha(n_1)}{(n_1+\alpha)^\sigma}\right)
+ \cdots +\supp \left(\frac{\mathbb{X}_\alpha(n_k)}{(n_k+\alpha)^\sigma}\right)
= \{ z :  a_N \leq |z| \leq b_N \}, 
\end{gather*}
where $a_N$ and $b_N$ are non-negative real numbers determined as follows. 
Assume that there exists an element $n_{j_0} \in \mathcal{K}_\alpha(N)$ such that the inequality
\begin{gather}\label{eq:10110016}
\frac{1}{(n_{j_0}+\alpha)^\sigma}
> \sum_{\substack{1 \leq j \leq k \\ j \neq j_0}} \frac{1}{(n_j+\alpha)^\sigma}
\end{gather}
is satisfied. 
In this case we put
\begin{gather*}
a_N
=\frac{1}{(n_{j_0}+\alpha)^\sigma}
- \sum_{\substack{1 \leq j \leq k \\ j \neq j_0}} \frac{1}{(n_j+\alpha)^\sigma}. 
\end{gather*}
Otherwise, we put $a_N=0$. 
Also, the non-negative real number $b_N$ is determined as
\begin{gather*}
b_N
= \frac{1}{(n_1+\alpha)^\sigma}+ \cdots + \frac{1}{(n_k+\alpha)^\sigma}
= \sum_{n \in \mathcal{K}_\alpha(N)} \frac{1}{(n+\alpha)^\sigma}. 
\end{gather*}
Note that \eqref{eq:10110016} is equivalent to 
\begin{gather*}
\frac{2}{(n_{j_0}+\alpha)^\sigma}
> \sum_{n \in \mathcal{K}_\alpha(N)} \frac{1}{(n+\alpha)^\sigma}. 
\end{gather*}
However, we have $2/(n_{j_0}+\alpha)^\sigma \ll N^{-\sigma}$ due to $n_{j_0} \in (N, N \log{N}]$, while the right-hand side has the lower bound
\begin{gather*}
\sum_{n \in \mathcal{K}_\alpha(N)} \frac{1}{(n+\alpha)^\sigma}
> \frac{51}{100} \sum_{n \in \mathcal{L}(N)} \frac{1}{(n+\alpha)^\sigma}
\gg_\sigma (N \log{N})^{1-\sigma}
\end{gather*}
for $\alpha \in \mathcal{A}_{d}$ and $N \geq N_1=N_1(d,\sigma)$ by Proposition \ref{prop:4.1}. 
Therefore the case in which \eqref{eq:10110016} is valid does not occur for $N \geq N_1$. 
Hence \eqref{eq:10110025} yields
\begin{gather}\label{eq:10110052}
\supp \left( \sum_{n \in \mathcal{M}(N)} \frac{\mathbb{X}_\alpha(n)}{(n+\alpha)^\sigma} \right)
= \{z : |z| \leq b_N \} 
+ \supp \left(\sum_{j=k+1}^{\ell} \frac{\mathbb{X}_\alpha(n_j)}{(n_j+\alpha)^\sigma}\right)
\end{gather}
for $N \geq N_1$.
Then we prove \eqref{eq:10110027} by using this formula. 
Let $z_0 \in \mathbb{C}$. 
We take an arbitrarily complex number $w_0$ such that 
\begin{gather*}
w_0
\in \supp \left(\sum_{j=k+1}^{\ell} \frac{\mathbb{X}_\alpha(n_j)}{(n_j+\alpha)^\sigma}\right).  
\end{gather*}
Then the absolute value of $w_0$ is bounded as
\begin{gather}\label{eq:10110201}
|w_0| 
\leq \sum_{j=k+1}^{\ell} \frac{1}{(n_j+\alpha)^\sigma}
= \sum_{n=0}^{N} \frac{1}{(n+\alpha)^\sigma}
+ \sum_{n \in \mathcal{L}(N) \setminus \mathcal{K}_\alpha(N)} \frac{1}{(n+\alpha)^\sigma}. 
\end{gather}
Note that we have 
\begin{gather*}
\sum_{n=0}^{N} \frac{1}{(n+\alpha)^\sigma}
\asymp_{d,\sigma} N^{1-\sigma}
\quad\text{and}\quad
\sum_{n \in \mathcal{L}(N)} \frac{1}{(n+\alpha)^\sigma}
\asymp_\sigma (N \log{N})^{1-\sigma}
\end{gather*}
for $\alpha \in \mathcal{A}_{d}$ since $\alpha>d^{-1}$ holds by the definition of $\mathcal{A}_{d}$. 
Thus we see that 
\begin{gather*}
\sum_{n=0}^{N} \frac{1}{(n+\alpha)^\sigma}
\leq \frac{1}{200} \sum_{n \in \mathcal{L}(N)} \frac{1}{(n+\alpha)^\sigma}
\end{gather*}
holds for any $N \geq N_{2,1}$, where $N_{2,1}=N_{2,1}(d,\sigma)$ is an integer depending only on $d$ and $\sigma$. 
By Proposition \ref{prop:4.1}, we also obtain
\begin{gather*}
\sum_{n \in \mathcal{L}(N) \setminus \mathcal{K}_\alpha(N)} \frac{1}{(n+\alpha)^\sigma}
< \frac{49}{100} \sum_{n \in \mathcal{L}(N)} \frac{1}{(n+\alpha)^\sigma}
\end{gather*}
for $N \geq N_1$. 
Combining these inequalities, we deduce from \eqref{eq:10110201} that
\begin{gather*}
|w_0| 
< \frac{99}{200} \sum_{n \in \mathcal{L}(N)} \frac{1}{(n+\alpha)^\sigma}
\end{gather*}
is satisfied for any $N \geq \max\{N_1, N_{2,1}\}$. 
On the other hand, there exists an integer $N_{2,2}=N_{2,2}(\sigma,z_0)$ depending only on $\sigma$ and $z_0$ such that 
\begin{gather*}
|z_0|
\leq \frac{1}{200} \sum_{n \in \mathcal{L}(N)} \frac{1}{(n+\alpha)^\sigma}
\end{gather*}
for any $N \geq N_{2,2}$. 
From these, we obtain
\begin{gather*}
|z_0-w_0|
\leq \frac{50}{100} \sum_{n \in \mathcal{L}(N)} \frac{1}{(n+\alpha)^\sigma}
< \sum_{n \in \mathcal{K}_\alpha(N)} \frac{1}{(n+\alpha)^\sigma}
= b_N 
\end{gather*}
for $N \geq \max\{N_1, N_{2,1}, N_{2,2}\}$ by applying Proposition \ref{prop:4.1} again. 
This implies that $z_0-w_0 \in \{z : |z| \leq b_N \}$, and furthermore, we deduce from \eqref{eq:10110052} that
\begin{gather*}
z_0
= (z_0-w_0)+w_0
\in \supp \left( \sum_{n \in \mathcal{M}(N)} \frac{\mathbb{X}_\alpha(n)}{(n+\alpha)^\sigma} \right). 
\end{gather*}
Therefore we obtain \eqref{eq:10110027} by changing variables. 
\end{proof}

\subsection{A restricted expectation}\label{sec:6.2}
Let $\mathcal{X}$ be a random variable defined on the space $(\Omega, \mathcal{F}, \mathbf{P})$. 
For an event $\Omega_0 \in \mathcal{F}$, we define
\begin{gather*}
\mathbf{E}\left[\mathcal{X} : \Omega_0\right]
= \int_{\Omega_0} \mathcal{X}(\omega) \,\mathbf{P}(d \omega). 
\end{gather*}
In this subsection, we evaluate $\mathbf{E}\left[|\zeta(\sigma, \mathbb{X}_\alpha)-\zeta_N(\sigma, \mathbb{X}_\alpha)|^2 : \Omega_0(N) \right]$ with an event $\Omega_0(N)$ defined as follows. 
By the definition of the support, \eqref{eq:10110027} implies that, for any $\epsilon>0$, we can take at least one element $\omega_0 \in \Omega$ such that $|\zeta_N(\sigma, \mathbb{X}_\alpha)(\omega_0) -z_0|<\epsilon$ is satisfied. 
Put $\theta_n=\arg \mathbb{X}_\alpha(n)(\omega_0)$ for $0 \leq n \leq N$. 
Then we define
\begin{align*}
\Omega_0(N)
&= \Omega_0(N; \alpha,\delta) \\
&= \left\{\omega \in \Omega ~\middle|~ 
\text{$\mathbb{X}_\alpha(n)(\omega) \in A(\theta_n-2\pi \delta,\theta_n+2\pi \delta)$ for all $0 \leq n \leq N$}\right\}
\end{align*}
for $0<\delta<1/2$, where $A(s,t)$ denotes the arc as in \eqref{eq:10121028}. 
Denote by $\mathbf{1}_{A(s,t)}$ the indicator function of $A(s,t)$. 
We begin by showing the following preliminary lemmas. 

\begin{lemma}\label{lem:5.2}
Let $s,t \in \mathbb{R}$ with $0<t-s \leq 2\pi$, and let $\Delta>1$.  
Then there exist Laurent polynomials
\begin{gather}\label{eq:10110703}
\mathscr{U}_{s,t}(z, \Delta)
= \sum_{|m| \leq \Delta} \widetilde{U}_{s,t}(m, \Delta) z^m, 
\quad
\mathscr{K}_{s,t}(z, \Delta)
= \sum_{|m| \leq \Delta} \widetilde{K}_{s,t}(m, \Delta) z^m 
\end{gather}
such that the inequality
\begin{gather}\label{eq:10100341}
|\mathbf{1}_{A(s,t)}(z) - \mathscr{U}_{s,t}(z, \Delta)| 
\leq \mathscr{K}_{s,t}(z, \Delta) 
\end{gather}
holds for $z \in \mathbb{C}$ with $|z|=1$, where $\widetilde{U}_{s,t}(m, \Delta)$ and $\widetilde{K}_{s,t}(m, \Delta)$ satisfy
\begin{align}\label{eq:03252312}
\widetilde{U}_{s,t}(0, \Delta)
= \frac{t-s}{2\pi}, 
\quad
\widetilde{U}_{s,t}(m, \Delta)
\ll \frac{1}{|m|} 
\quad (m \neq0), 
\quad\text{and}\quad
\widetilde{K}_{s,t}(0, \Delta)
\ll \frac{1}{\Delta}  
\end{align}
with absolute implied constants. 
\end{lemma}

\begin{proof}
Let $\psi(x)$ denote the saw-tooth function presented by
\begin{align*}
\psi(x)
= 
\begin{cases}
x-[x]-\frac{1}{2} & \text{if $x \notin \mathbb{Z}$}, 
\\
0 & \text{if $x \in \mathbb{Z}$}.
\end{cases}
\end{align*}
By definition, $\mathbf{1}_{A(s,t)}(z)$ and $\psi(x)$ are related as
\begin{align}\label{eq:03252255}
\mathbf{1}_{A(s,t)}(z)
= \frac{t-s}{2\pi} 
+ \psi \left(\frac{s-\arg(z)}{2\pi}\right)
+ \psi \left(\frac{\arg(z)-t}{2\pi}\right)
\end{align}
if $\arg(z)-s, \arg(z)-t \notin 2\pi \mathbb{Z}$, see also \cite[$(2.5)$]{BartonMontgomeryVaaler2001}. 
For an integrable function $F$ on $\mathbb{R}$, we put $F_\delta(x)= \delta F(\delta x)$. 
Then the Fourier transform of $F_\delta$ satisfies $\hat{F}_\delta(t)=\hat{F}(\delta^{-1}t)$ for $\delta>0$. 
Let $J(z)$ and $K(z)$ be the entire functions as in \cite[Section 2]{Vaaler1985}. 
Then we introduce the trigonometric polynomials
\begin{align*}
j_\Delta(x)
= \sum_{|m| \leq \Delta} \hat{J}_{\Delta+1}(m) e^{2\pi imx}
\quad\text{and}\quad
k_\Delta(x)
= \sum_{|m| \leq \Delta} \hat{K}_{\Delta+1}(m) e^{2\pi imx} 
\end{align*}
according to \cite[Section 7]{Vaaler1985}. 
Furthermore, if we define 
\begin{align*}
\psi*j_\Delta(x)
= \int_{-1/2}^{1/2} \psi(x-\xi) j_\Delta(\xi) \,d \xi
= \sum_{\substack{|m| \leq \Delta \\ m \neq0}} \frac{\hat{J}_{\Delta+1}(m)}{-2\pi im} e^{2\pi imx}, 
\end{align*}
then we deduce from \cite[Theorem 18]{Vaaler1985} the inequality 
\begin{align}\label{eq:03252256}
\left|\psi(x)-\psi*j_\Delta(x)\right|
\leq (2\Delta+2)^{-1} k_\Delta(x). 
\end{align}
Therefore, inequality \eqref{eq:10100341} can be deduced from \eqref{eq:03252255} and \eqref{eq:03252256} if we assume $\arg(z)-s, \arg(z)-t \notin 2\pi \mathbb{Z}$, where 
\begin{align*}
\widetilde{U}_{s,t}(m, \Delta)
&= 
\begin{cases}
\displaystyle{\frac{t-s}{2\pi}}
&\quad (m=0), 
\\[2mm]
\displaystyle{ \frac{\hat{J}_{\Delta+1}(-m) e^{-ims} - \hat{J}_{\Delta+1}(m) e^{-imt}}{2\pi i m}}
&\quad (m \neq0), 
\end{cases}
\\
\widetilde{K}_{s,t}(m, \Delta)
&= \frac{\hat{K}_{\Delta+1}(-m) e^{-ims} + \hat{K}_{\Delta+1}(m) e^{-imt}}{2\Delta+2}. 
\end{align*}
Furthermore, we see that \eqref{eq:10100341} holds for any $z$ by the continuity of $\mathscr{U}_{s,t}(z, \Delta)$ and $\mathscr{K}_{s,t}(z, \Delta)$. 
Recall that $\hat{J}_{\Delta+1}(m)$ and $\hat{K}_{\Delta+1}(m)$ are calculated as
\begin{align*}
\hat{J}_{\Delta+1}(m)
&= \frac{\pi m}{\Delta+1} \left(1-\frac{|m|}{\Delta+1}\right) \cot \frac{\pi m}{\Delta+1}
+ \frac{|m|}{\Delta+1}
\quad (m \neq0), \\
\hat{K}_{\Delta+1}(m)
&= 1-\frac{|m|}{\Delta+1}
\end{align*}
for $|m| \leq \Delta$ by \cite[p.~339]{BartonMontgomeryVaaler2001}. 
Hence \eqref{eq:03252312} immediately follows. 
\end{proof}

\begin{lemma}\label{lem:5.3}
Let $s_n,t_n \in \mathbb{R}$ with $0<t_n-s_n \leq 2\pi$ for $0 \leq n \leq N$, and let $\Delta \geq3$. 
Then we have 
\begin{align}\label{eq:10100402}
\left| \prod_{n=0}^{N} \mathbf{1}_{A(s_n,t_n)}(z_n) 
- \prod_{n=0}^{N} \mathscr{U}_{s_n,t_n}(z_n, \Delta) \right| 
\ll (\log{\Delta})^{N+1} \sum_{n=0}^{N} \mathscr{K}_{s_n,t_n}(z_n, \Delta) 
\end{align}
for $(z_0,\ldots,z_N) \in \mathbb{\mathbb{C}}^{N+1}$ with $|z_0|=\cdots=|z_N|=1$, where $\mathscr{U}_{s,t}(z, \Delta)$ and $\mathscr{K}_{s,t}(z, \Delta)$ are as in Lemma \ref{lem:5.2}, and the implied constant is absolute. 
\end{lemma}

\begin{proof}
Note that it is deduced from \eqref{eq:03252312} the estimate
\begin{gather}\label{eq:10102249}
\mathscr{U}_{s,t}(z, \Delta)
\ll 1+\sum_{0<m \leq \Delta} \frac{1}{m}
\ll \log{\Delta}. 
\end{gather}
Then we prove \eqref{eq:10100402} by induction on $N$. 
Firstly, Lemma \ref{lem:5.2} yields the result for $N=0$. 
Furthermore, we have 
\begin{align*}
&\left| \prod_{n=0}^{N+1} \mathbf{1}_{A(s_n,t_n)}(z_n) 
- \prod_{n=0}^{N+1} \mathscr{U}_{s_n,t_n}(z_n, \Delta) \right| \\
&\leq \prod_{n=0}^{N} \mathbf{1}_{A(s_n,t_n)}(z_n) 
\cdot \Big| \mathbf{1}_{A(s_{N+1},t_{N+1})}(z_{N+1}) - \mathscr{U}_{s_{N+1},t_{N+1}}(z_{N+1}, \Delta) \Big| \\
&\qquad
+ \left| \prod_{n=0}^{N} \mathbf{1}_{A(s_n,t_n)}(z_n) 
- \prod_{n=0}^{N} \mathscr{U}_{s_n,t_n}(z_n, \Delta) \right| 
\cdot \Big|\mathscr{U}_{s_{N+1},t_{N+1}}(z_{N+1}, \Delta)\Big|. 
\end{align*}
By the inductive assumption and \eqref{eq:10102249}, we obtain the result for $N+1$. 
\end{proof}

Applying Lemma \ref{lem:5.3}, we obtain a good approximation of the indicator function $\mathbf{1}_{\Omega_0(N)}(\omega)$. 
Then we prove the following propositions which are used to complete the proof of Theorem \ref{thm:1.3}. 

\begin{proposition}\label{prop:6.2}
Let $N_2$ be the integer as in Proposition \ref{prop:6.1}. 
For any integer $N \geq N_2$ and $\Delta \geq \Delta_0$ with some absolute constant $\Delta_0$, there exists a finite subset $\mathcal{E}_{d}^{(1)}=\mathcal{E}_{d}^{(1)}(N, \Delta) \subset \mathcal{A}_{d}$ such that 
\begin{gather*}
\mathbf{P}\left(\Omega_0(N)\right)
= (2\delta)^{N+1}
+ O\left(\frac{N (\log{\Delta})^{N+1}}{\Delta}\right)
\end{gather*}
holds for all $\alpha \in \mathcal{A}_{d} \setminus \mathcal{E}_{d}^{(1)}$, where the implied constant is absolute. 
\end{proposition}

\begin{proof}
Put $s_n=\theta_n-2\pi\delta$ and $t_n=\theta_n+2\pi\delta$. 
Then we have 
\begin{align}\label{eq:10110748}
&\mathbf{P}\left(\Omega_0(N)\right)
= \mathbf{E}\left[\prod_{n=0}^{N} \mathbf{1}_{A(s_n,t_n)}(\mathbb{X}_\alpha(n))\right] \\
&= \mathbf{E}\left[\prod_{n=0}^{N} \mathscr{U}_{s_n,t_n}(\mathbb{X}_\alpha(n), \Delta)\right]
+O\left((\log{\Delta})^{N+1} 
\sum_{n=0}^{N} \mathbf{E}[\mathscr{K}_{s_n,t_n}(\mathbb{X}_\alpha(n), \Delta)]\right)
\nonumber
\end{align}
by Lemma \ref{lem:5.3}. 
Inserting \eqref{eq:10110703}, we calculate the main term as
\begin{align}\label{eq:10110739}
&\mathbf{E}\left[\prod_{n=0}^{N} \mathscr{U}_{s_n,t_n}(\mathbb{X}_\alpha(n), \Delta)\right] \\
&= \sum_{|m_0| \leq \Delta} \cdots \sum_{|m_N| \leq \Delta}
\prod_{n=0}^{N} \widetilde{U}_{s_n,t_n}(m_n, \Delta) \,
\mathbf{E}\left[\prod_{n=0}^{N} \mathbb{X}_\alpha(n)^{m_n}\right]. 
\nonumber
\end{align}
Here, we have 
\begin{gather*}
\mathbf{E}\left[\prod_{n=0}^{N} \mathbb{X}_\alpha(n)^{m_n}\right]
=
\begin{cases}
1
& \text{if $\prod_{n=0}^{N}(n+\alpha)^{m_n}=1$},
\\
0
& \text{otherwise}
\end{cases}
\end{gather*}
by condition \eqref{eq:10010220}. 
Then we define $\mathcal{E}_{d}^{(1)}=\mathcal{E}_{d}^{(1)}(N, \Delta)$ as the set of all $\alpha \in \mathcal{A}_{d}$ such that $\prod_{n=0}^{N}(n+\alpha)^{m_n}=1$ for some $(m_0, \ldots, m_N) \in \mathbb{Z}^{N+1} \setminus \{0\}$ with $|m_n| \leq \Delta$ for any $0 \leq n \leq N$. 
Note that $\mathcal{E}_{d}^{(1)}$ is a finite set when $N$ and $\Delta$ are given. 
Furthermore, if $\alpha \in \mathcal{A}_{d} \setminus \mathcal{E}_{d}^{(1)}$, then the terms in \eqref{eq:10110739} vanish unless $m_0= \cdots =m_N=0$. 
Thus we obtain
\begin{gather}\label{eq:10110755}
\mathbf{E}\left[\prod_{n=0}^{N} \mathscr{U}_{s_n,t_n}(\mathbb{X}_\alpha(n), \Delta)\right] 
= \prod_{n=0}^{N} \widetilde{U}_{s_n,t_n}(0, \Delta) 
= (2\delta)^{N+1} 
\end{gather}
by using \eqref{eq:03252312}. 
Then we evaluate the error term in \eqref{eq:10110748}. 
We have 
\begin{gather*}
\mathbf{E}[\mathscr{K}_{s_n,t_n}(\mathbb{X}_\alpha(n), \Delta)]
= \sum_{|m| \leq \Delta} \widetilde{K}_{s_n,t_n}(m, \Delta) 
\mathbf{E}\left[\mathbb{X}_\alpha(n)^m \right]
= \widetilde{K}_{s_n,t_n}(0, \Delta)
\end{gather*}
since $\mathbf{E}\left[\mathbb{X}_\alpha(n)^m \right]=0$ for $m \neq 0$ by \eqref{eq:10010220}. 
Therefore, by \eqref{eq:03252312}, we see that 
\begin{gather}\label{eq:10110756}
(\log{\Delta})^{N+1} 
\sum_{n=0}^{N} \mathbf{E}[\mathscr{K}_{s_n,t_n}(\mathbb{X}_\alpha(n), \Delta)]
\ll \frac{N (\log{\Delta})^{N+1}}{\Delta}. 
\end{gather}
Combining \eqref{eq:10110755} and \eqref{eq:10110756}, we obtain the desired result. 
\end{proof}

\begin{proposition}\label{prop:6.3}
Let $N_2$ be the integer as in Proposition \ref{prop:6.1}. 
Let $1/2<\sigma<1$ be a fixed real number. 
For any integers $N,L$ satisfying $L>N \geq N_2$ and $\Delta \geq \Delta_0$ with some absolute constant $\Delta_0$, there exists a finite subset $\mathcal{E}_{d}^{(2)}=\mathcal{E}_{d}^{(2)}(N,L,\Delta) \subset \mathcal{A}_{d}$ such that 
\begin{align*}
&\mathbf{E}\left[|\zeta(\sigma, \mathbb{X}_\alpha)-\zeta_N(\sigma, \mathbb{X}_\alpha)|^2 : \Omega_0(N) \right] \\
&\ll \mathbf{P}(\Omega_0(N)) N^{1-2\sigma}
+ L^{1-2\sigma}
+ \frac{NL (\log{\Delta})^{N+1}}{\Delta}
\end{align*}
holds for all $\alpha \in \mathcal{A}_{d} \setminus \mathcal{E}_{d}^{(2)}$, where the implied constant depends only on $\sigma$. 
\end{proposition}

\begin{proof}
By the Cauchy--Schwarz inequality, we have 
\begin{align}\label{eq:10111050}
&\mathbf{E}\left[|\zeta(\sigma, \mathbb{X}_\alpha)-\zeta_N(\sigma, \mathbb{X}_\alpha)|^2 : \Omega_0(N) \right] \\
&\ll \mathbf{E}\left[|\zeta_L(\sigma, \mathbb{X}_\alpha)-\zeta_N(\sigma, \mathbb{X}_\alpha)|^2 : \Omega_0(N) \right] \nonumber \\
&\qquad
+ \mathbf{E}\left[|\zeta(\sigma, \mathbb{X}_\alpha)-\zeta_L(\sigma, \mathbb{X}_\alpha)|^2 : \Omega_0(N) \right]. 
\nonumber
\end{align}
The first term is calculated as
\begin{align*}
&\mathbf{E}\left[|\zeta_L(\sigma, \mathbb{X}_\alpha)-\zeta_N(\sigma, \mathbb{X}_\alpha)|^2 : \Omega_0(N) \right] \\
&= \mathbf{P}\left(\Omega_0(N)\right) 
\sum_{N< n \leq L} \frac{1}{(n+\alpha)^{2\sigma}}
+ \sum_{\substack{N<n_1, n_2 \leq L \\ n_1 \neq n_2}}
\frac{\mathbf{E}[\mathbb{X}_\alpha(n_1) \overline{\mathbb{X}_\alpha(n_2)} : \Omega_0(N)]}
{(n_1+\alpha)^\sigma (n_2+\alpha)^\sigma}
\end{align*}
since $\mathbf{E}[1 : \Omega_0(N)] = \mathbf{P}\left(\Omega_0(N)\right)$ by definition. 
We have 
\begin{gather}\label{eq:10111044}
\mathbf{P}\left(\Omega_0(N)\right) 
\sum_{N< n \leq L} \frac{1}{(n+\alpha)^{2\sigma}}
\ll_\sigma \mathbf{P}\left(\Omega_0(N)\right) N^{1-2\sigma}. 
\end{gather}
Furthermore, we apply Lemma \ref{lem:5.3} to obtain
\begin{gather}\label{eq:10111043}
\sum_{\substack{N<n_1, n_2 \leq L \\ n_1 \neq n_2}}
\frac{\mathbf{E}[\mathbb{X}_\alpha(n_1) \overline{\mathbb{X}_\alpha(n_2)} : \Omega_0(N)]}
{(n_1+\alpha)^\sigma (n_2+\alpha)^\sigma}
= S_1+S_2, 
\end{gather}
where
\begin{align*}
S_1
&= \sum_{\substack{N<n_1, n_2 \leq L \\ n_1 \neq n_2}}
\frac{1}{(n_1+\alpha)^\sigma (n_2+\alpha)^\sigma}
\mathbf{E}\left[\mathbb{X}_\alpha(n_1) \overline{\mathbb{X}_\alpha(n_2)} 
\prod_{n=0}^{N} \mathscr{U}_{s_n,t_n}(\mathbb{X}_\alpha(n), \Delta)\right], \\
S_2
&\ll \sum_{\substack{N<n_1, n_2 \leq L \\ n_1 \neq n_2}}
\frac{1}{(n_1+\alpha)^\sigma (n_2+\alpha)^\sigma}
(\log{\Delta})^{N+1}
\sum_{n=0}^{N} \mathbf{E}[\mathscr{K}_{s_n,t_n}(\mathbb{X}_\alpha(n), \Delta)]. 
\end{align*}
Similarly to \eqref{eq:10110739}, we have 
\begin{align*}
&\mathbf{E}\left[\mathbb{X}_\alpha(n_1) \overline{\mathbb{X}_\alpha(n_2)}
\prod_{n=0}^{N} \mathscr{U}_{s_n,t_n}(\mathbb{X}_\alpha(n), \Delta)\right] \\
&= \sum_{|m_0| \leq \Delta} \cdots \sum_{|m_N| \leq \Delta}
\prod_{n=0}^{N} \widetilde{U}_{s_n,t_n}(m_n, \Delta) \,
\mathbf{E}\left[\mathbb{X}_\alpha(n_1) \overline{\mathbb{X}_\alpha(n_2)} 
\prod_{n=0}^{N} \mathbb{X}_\alpha(n)^{m_n}\right]. 
\nonumber
\end{align*}
Then we define $\mathcal{E}_{d}^{(2)}=\mathcal{E}_{d}^{(2)}(N,L,\Delta)$ as the set of all $\alpha \in \mathcal{A}_{d}$ such that 
\begin{gather*}
\frac{n_2+\alpha}{n_1+\alpha}
= \prod_{n=0}^{N} (n+\alpha)^{m_n}
\end{gather*}
for some $N<n_1,n_2 \leq L$ with $n_1 \neq n_2$ and $(m_0, \ldots, m_N) \in \mathbb{Z}^{N+1}$ with $|m_n| \leq \Delta$ for any $0 \leq n \leq N$. 
We see that $\mathcal{E}_{d}^{(2)}$ is a finite set when $N, L$, and $\Delta$ are given. 
By the definition, if $\alpha \in \mathcal{A}_{d} \setminus \mathcal{E}_{d}^{(2)}$, then we have 
\begin{gather*}
\mathbf{E}\left[\mathbb{X}_\alpha(n_1) \overline{\mathbb{X}_\alpha(n_2)} 
\prod_{n=0}^{N} \mathbb{X}_\alpha(n)^{m_n}\right]
= 0
\end{gather*}
for any $N<n_1,n_2 \leq L$ with $n_1 \neq n_2$ and $(m_0, \ldots, m_N) \in \mathbb{Z}^{N+1}$ with $|m_n| \leq \Delta$. 
Hence we conclude that $S_1=0$ for $\alpha \in \mathcal{A}_{d} \setminus \mathcal{E}_{d}^{(2)}$. 
Next, we have 
\begin{gather*}
\sum_{\substack{N<n_1, n_2 \leq L \\ n_1 \neq n_2}}
\frac{1}{(n_1+\alpha)^\sigma (n_2+\alpha)^\sigma}
\leq \left(\sum_{N<n \leq L} \frac{1}{(n+\alpha)^{1/2}} \right)^2
\ll L. 
\end{gather*}
As checked before, we also obtain $\mathbf{E}[\mathscr{K}_{s_n,t_n}(\mathbb{X}_\alpha(n), \Delta)] \ll \Delta^{-1}$. 
Therefore, $S_2$ is estimated as 
\begin{gather*}
S_2
\ll \frac{NL (\log{\Delta})^{N+1}}{\Delta}. 
\end{gather*}
Thus \eqref{eq:10111043} yields 
\begin{gather*}
\sum_{\substack{N<n_1, n_2 \leq L \\ n_1 \neq n_2}}
\frac{\mathbf{E}[\mathbb{X}_\alpha(n_1) \overline{\mathbb{X}_\alpha(n_2)} : \Omega_0(N)]}
{(n_1+\alpha)^\sigma (n_2+\alpha)^\sigma}
\ll \frac{NL (\log{\Delta})^{N+1}}{\Delta}
\end{gather*}
for $\alpha \in \mathcal{A}_{d} \setminus \mathcal{E}_{d}^{(2)}$. 
Then, together with \eqref{eq:10111044}, we obtain the upper bound
\begin{gather*}
\mathbf{E}\left[|\zeta_L(\sigma, \mathbb{X}_\alpha)-\zeta_N(\sigma, \mathbb{X}_\alpha)|^2 : \Omega_0(N) \right]
\ll \mathbf{P}(\Omega_0(N)) N^{1-2\sigma}
+ \frac{NL (\log{\Delta})^{N+1}}{\Delta}. 
\end{gather*}
The estimate for the second term in \eqref{eq:10111050} remains. 
We have 
\begin{align*}
&\mathbf{E}\left[|\zeta(\sigma, \mathbb{X}_\alpha)-\zeta_L(\sigma, \mathbb{X}_\alpha)|^2 : \Omega_0(N) \right] \\
&\leq \mathbf{E}\left[|\zeta(\sigma, \mathbb{X}_\alpha)-\zeta_L(\sigma, \mathbb{X}_\alpha)|^2 \right] 
= \sum_{n>L} \frac{1}{(n+\alpha)^{2\sigma}}
\ll_\sigma L^{1-2\sigma}. 
\end{align*}
From these, we obtain the conclusion. 
\end{proof}

\subsection{Completion of the proof}\label{sec:6.3}
By Propositions \ref{prop:6.1}, \ref{prop:6.2}, and \ref{prop:6.3}, we obtain the following result. 

\begin{corollary}\label{cor:6.4}
Let $1/2<\sigma<1$ be a fixed real number. 
Then, for any $z_0 \in \mathbb{C}$ and $\epsilon>0$, there exists a finite subset $\mathcal{E}_{d}=\mathcal{E}_{d}(\sigma,z_0,\epsilon) \subset \mathcal{A}_{d}$ such that 
\begin{gather*}
\mathbf{P}\left(|\zeta(\sigma, \mathbb{X}_\alpha) -z_0|<\epsilon\right)
>0
\end{gather*}
holds for any $\alpha \in \mathcal{A}_{d} \setminus \mathcal{E}_{d}$.  
\end{corollary}

\begin{proof}
Let $N \geq N_2$ with the integer $N_2=N_2(d,\sigma,z_0)$ as in Proposition \ref{prop:6.1}. 
For $\omega \in \Omega_0(N)$, the difference between $\zeta_N(\sigma, \mathbb{X}_\alpha)(\omega)$ and $\zeta_N(\sigma, \mathbb{X}_\alpha)(\omega_0)$ is estimated as
\begin{align*}
\left|\zeta_N(\sigma, \mathbb{X}_\alpha)(\omega)-\zeta_N(\sigma, \mathbb{X}_\alpha)(\omega_0)\right|
&\leq \sum_{n=0}^{N} \frac{|\mathbb{X}_\alpha(n)(\omega)-\mathbb{X}_\alpha(n)(\omega_0)|}{(n+\alpha)^\sigma} \\
&\ll \delta \, \sum_{n=0}^{N} \frac{1}{(n+\alpha)^\sigma}
\ll_d \delta \, N^{1/2}
\end{align*}
for any $\alpha \in \mathcal{A}_{d}$, where the implied constants depend at most on $d$. 
Therefore, the condition $\omega \in \Omega_0(N)$ implies
\begin{gather*}
|\zeta_N(\sigma, \mathbb{X}_\alpha)(\omega) -z_0|
< 2\epsilon
\end{gather*}
for any $\alpha \in \mathcal{A}_{d}$ if we suppose $\delta \leq A(d) \epsilon \, N^{-1/2}$, where $A(d)$ is a suitable positive constant. 
Hence we have 
\begin{align*}
&\mathbf{P}\left(|\zeta(\sigma, \mathbb{X}_\alpha) -z_0|<3\epsilon\right) \\
&\geq \mathbf{P}\big(\Omega_0(N) \cap \{|\zeta(\sigma, \mathbb{X}_\alpha)-\zeta_N(\sigma, \mathbb{X}_\alpha)|< \epsilon\} \big) \\
&= \mathbf{P}\big(\Omega_0(N)\big)
- \mathbf{P}\big(\Omega_0(N) \cap \{|\zeta(\sigma, \mathbb{X}_\alpha)-\zeta_N(\sigma, \mathbb{X}_\alpha)| \geq \epsilon\} \big)
\end{align*}
if $\delta \leq A(d) \epsilon \, N^{-1/2}$ is satisfied. 
Furthermore, Proposition \ref{prop:6.3} yields
\begin{align*}
&\mathbf{P}\big(\Omega_0(N) \cap \{|\zeta(\sigma, \mathbb{X}_\alpha)-\zeta_N(\sigma, \mathbb{X}_\alpha)| \geq \epsilon\} \big) \\
&\leq \frac{1}{\epsilon^2} 
\mathbf{E}\left[|\zeta(\sigma, \mathbb{X}_\alpha)-\zeta_N(\sigma, \mathbb{X}_\alpha)|^2 : \Omega_0(N) \right] \\
&\leq \frac{B(\sigma)}{\epsilon^2} \left\{\mathbf{P}(\Omega_0(N)) N^{1-2\sigma}
+ L^{1-2\sigma}
+ \frac{NL (\log{\Delta})^{N+1}}{\Delta} \right\}
\end{align*}
for any $\alpha \in \mathcal{A}_{d} \setminus \mathcal{E}_{d}^{(2)}$, where $B(\sigma)$ is a positive constant depending only on $\sigma$. 
Here, the set $\mathcal{E}_{d}^{(2)}=\mathcal{E}_{d}^{(2)}(N,L,\Delta)$ is as in Proposition \ref{prop:6.3}. 
We can take a positive integer $N_3=N_3(\sigma,\epsilon,N_2) \geq N_2$ such that 
\begin{gather*}
\frac{B(\sigma)}{\epsilon^2} \mathbf{P}(\Omega_0(N_3)) N_3^{1-2\sigma}
< \frac{1}{2} \mathbf{P}(\Omega_0(N_3)) 
\end{gather*}
is satisfied. 
Then, putting $N=N_3$ and $\delta=A(d) \epsilon \, N_3^{-1/2}$, we derive the inequality
\begin{align*}
&\mathbf{P}\left(|\zeta(\sigma, \mathbb{X}_\alpha) -z_0|<3\epsilon\right) \\
&> \frac{1}{2} \mathbf{P}(\Omega_0(N_3))
- \frac{B(\sigma)}{\epsilon^2} \left\{L^{1-2\sigma}
+ \frac{N_3 L (\log{\Delta})^{N_3+1}}{\Delta} \right\}
\end{align*}
for $\alpha \in \mathcal{A}_{d} \setminus \mathcal{E}_{d}^{(2)}$. 
Furthermore, we apply Proposition \ref{prop:6.2} to deduce
\begin{gather*}
\mathbf{P}\left(|\zeta(\sigma, \mathbb{X}_\alpha) -z_0|<3\epsilon\right)
> (2 \delta)^{N_3+1}
- \frac{B'(\sigma)}{\epsilon^2} \left\{L^{1-2\sigma}
+ \frac{N_3 L (\log{\Delta})^{N_3+1}}{\Delta} \right\}
\end{gather*}
for $\alpha \in \mathcal{A}_{d} \setminus (\mathcal{E}_{d}^{(1)} \cup \mathcal{E}_{d}^{(2)})$, where $B'(\sigma)$ is some positive constant, and the set $\mathcal{E}_{d}^{(1)}=\mathcal{E}_{d}^{(1)}(N_3, \Delta)$ is as in Proposition \ref{prop:6.2}. 
Then, we see that there exists an integer $L=L(\sigma,\epsilon,\delta,N_3)$ satisfying 
\begin{gather*}
\frac{B'(\sigma)}{\epsilon^2} L^{1-2\sigma}
< \frac{1}{4} (2 \delta)^{N_3+1}. 
\end{gather*}
Finally, we take a real number $\Delta=\Delta(\sigma,\epsilon,\delta,N_3) \geq3$ so that
\begin{gather*}
\frac{B'(\sigma)}{\epsilon^2}  \frac{N_3 L (\log{\Delta})^{N_3+1}}{\Delta}
< \frac{1}{4} (2 \delta)^{N_3+1}
\end{gather*}
is satisfied. 
Then we obtain
\begin{gather*}
\mathbf{P}\left(|\zeta(\sigma, \mathbb{X}_\alpha) -z_0|<3\epsilon\right)
> \frac{1}{2} (2 \delta)^{N_3+1}
\end{gather*}
for $\alpha \in \mathcal{A}_{d} \setminus \mathcal{E}_{d}$ with $\mathcal{E}_{d}= \mathcal{E}_{d}^{(1)} \cup \mathcal{E}_{d}^{(2)}$, which yields the desired conclusion. 
\end{proof}

\begin{proof}[Proof of Theorem \ref{thm:1.3}]
As we have seen in Example \ref{exa:a}, the set $\mathcal{A}_d$ contains infinitely many algebraic irrational numbers for $d \geq5$. 
By the Portmanteau theorem \cite[Theorem 29.1]{Billingsley1995}, we deduce from Theorem \ref{thm:1.2} the inequality
\begin{align*}
&\liminf_{T \to\infty} 
\frac{1}{T} \meas \left\{t \in [0,T] ~\middle|~ |\zeta(\sigma+it,\alpha)-z_0|<\epsilon \right\} \\
&\geq \mathbf{P}\left(|\zeta(\sigma, \mathbb{X}_\alpha) -z_0|<\epsilon\right)
\end{align*}
for any $\alpha \in \mathcal{A}$. 
Hence Corollary \ref{cor:6.4} yields Theorem \ref{thm:1.3}. 
\end{proof}


\providecommand{\bysame}{\leavevmode\hbox to3em{\hrulefill}\thinspace}
\providecommand{\MR}{\relax\ifhmode\unskip\space\fi MR }
\providecommand{\MRhref}[2]{%
  \href{http://www.ams.org/mathscinet-getitem?mr=#1}{#2}
}
\providecommand{\href}[2]{#2}

\end{document}